\documentclass[a4paper,11pt,reqno]{amsart}

\usepackage{amssymb}
\usepackage[all]{xy}

\usepackage{hyperref}
\usepackage{xcolor}
\usepackage{enumitem}
\usepackage{subcaption}
\usepackage[final]{graphicx}
\usepackage{float}
\usepackage{epic}
\usepackage{setspace}

\usepackage[all]{xy}
\usepackage{color}

\usepackage{verbatim}

\usepackage{tikz}
\usetikzlibrary[topaths]

\newcount\mycount

\numberwithin{equation}{section}

\numberwithin{equation}{section}
\newtheorem{theorem}{Theorem}[section]

\newtheorem{lemma}[theorem]{Lemma}

\newtheorem{proposition}[theorem]{Proposition}

\newtheorem{corollary}[theorem]{Corollary}
\newtheorem*{theorem*}{Theorem}
\newtheorem*{question*}{Question}
\newtheorem*{corollary*}{Corollary}

\theoremstyle{definition}
\newtheorem{definition}[theorem]{Definition}

\theoremstyle{remark}
\newtheorem{example}[theorem]{Example}

\theoremstyle{remark}

\newtheorem{remark}[theorem]{Remark}

\begin{document}

\title{Crossed products as compact quantum metric spaces}

\author{Mario Klisse}

\address{KU Leuven, Department of Mathematics,
Celestijnenlaan 200B, 3001 Leuven, Belgium}

\email{mario.klisse@kuleuven.be}

\begin{abstract}

By employing the external Kasparov product, in \cite{HSWZ13},  Hawkins, Skalski, White, and Zacharias constructed spectral triples on crossed product C$^\ast$-algebras by equicontinuous actions of discrete groups. They further raised the question for whether their construction turns the respective crossed product into a compact quantum metric space in the sense of Rieffel. By introducing the concept of groups separated with respect to a given length function, we give an affirmative answer in the case of virtually Abelian groups equipped with certain orbit metric length functions. We further complement our results with a discussion of natural examples such as generalized Bunce-Deddens algebras and higher-dimensional non-commutative tori.

\end{abstract}

\date{\today. \emph{MSC2010:}  46L05, 46L87, 46L89, 47L65, 58B34. The author is supported by FWO research project G090420N of the Research Foundation Flanders.}

\maketitle

\section*{Introduction}

\vspace{3mm}

The standard Dirac operator
of a compact spin manifold encodes large parts of its geometrical
structure. Motivated by this, Connes introduced the notion
of spectral triples: a \emph{spectral triple} $(\mathcal{A},\mathcal{H},D)$
on a separable unital C$^{\ast}$-algebra $A$ consists of a norm
dense unital $\ast$-subalgebra $\mathcal{A}$ of $A$ that is boundedly
represented on a Hilbert space $\mathcal{H}$, and a densely defined
self-adjoint operator $D$ on $\mathcal{H}$ that has compact resolvent
and for which all commutators of elements in $\mathcal{A}$ with $D$ extend
to bounded operators. It is \emph{even}, if the triple carries the
additional structure of a $\mathbb{Z}_{2}$-grading. The concept (also
referred to as \emph{unbounded Fredholm modules}) is one of the fundamental
building blocks in the theory of non-commutative geometry.

Following Connes (see \cite{Connes89}), with a given spectral triple, one
can associate a pseudo-metric on the state space $\mathcal{S}(A)$
of $A$ via 
\begin{equation}
(\psi,\psi^{\prime})\mapsto\sup\{|\psi(a)-\psi^{\prime}(a) | \mid a\in\mathcal{A}\text{ with }\Vert[D,a]\Vert\leq1\}.\label{eq:PseudoMetric}
\end{equation}
This generalizes the Monge--Kantorovich metric on the space of probability
measures of a given compact metric space $X$ (see \cite{Kantorovitch45}), and
in this case, the induced topology coincides with the weak-$\ast$
topology. In the non-commutative setting, the latter statement does
not necessarily hold anymore. This observation inspired Rieffel to
introduce the notion of compact quantum metric spaces. Even though
the definition in \cite[Definition 2.2]{Rieffel04} is given in the general setting
of order unit spaces, in the present article, we will exclusively
be concerned with Lip-norms induced by spectral triples: for a spectral
triple $(\mathcal{A},\mathcal{H},D)$ on a C$^{\ast}$-algebra $A$,
the pair $(A,L_{D})$ with the \emph{Lipschitz semi-norm} $L_{D}:\mathcal{A} \ni a\mapsto\Vert[D,a]\Vert$
is called a compact quantum metric space if the pseudo-metric in \eqref{eq:PseudoMetric}
induces the weak-$\ast$ topology on the state space of $A$; in this
case, $L_{D}$ is called a \emph{Lip-norm}. \\

In \cite{CMRV08} and \cite{BMR10}, spectral triples on certain crossed products
by the integers were constructed from suitable triples on the corresponding
coefficient algebra. This approach was extended in \cite{HSWZ13},
where Hawkins, Skalski, White, and Zacharias make use of the external
Kasparov product to construct odd and even spectral triples on crossed
products by equicontinuous actions of discrete groups that are equipped
with a proper translation bounded function. Their construction translates
verbatim into the setting of groups equipped with proper length functions
as in \cite{Connes89}. Under the additional assumption that the
Lipschitz semi-norm induced by the original spectral triple on the
coefficient algebra provides a Lip-norm, the authors formulate the following
natural question, which they answer affirmatively for $G=\mathbb{Z}$
equipped with the word length function associated with the standard
generating set $\{-1,1\}$.

\begin{question*} \label{Question} Let $(\mathcal{A},\mathcal{H}_{A},D_{A})$
be a spectral triple on a separable unital C$^{\ast}$-algebra $A$
and assume that the induced Lipschitz semi-norm $L_{D_{A}}(a):=\left\Vert [D_{A},a]\right\Vert ,a\in\mathcal{A}$
defines a compact quantum metric space $(A,L_{D_{A}})$. Let further
$\alpha:G\rightarrow\text{Aut}(A)$ be a metrically equicontinuous
action of a discrete group $G$, equipped with a proper length function
$\ell:G\rightarrow\mathbb{R}_{+}$. Under what conditions does the
spectral triple defined in \cite{HSWZ13} define a compact quantum
metric space? \end{question*}
\noindent Similar questions were addressed in \cite{BMR10}
and \cite{KK21}, where the latter reference also provides a set of assumptions
ensuring that a continuous family of $\ast$-automorphisms of a compact
quantum metric space yields a field of crossed product algebras which
varies continuously in Rieffel's quantum Gromov-Hausdorff
distance.

In \cite{Rieffel02}, Rieffel examined quantum metric space structures
of (twisted) group C$^{\ast}$-algebras of free Abelian groups induced
by spectral triples coming from word length functions and restrictions
of norms on Euclidean spaces. His results were later extended to word
hyperbolic groups (see \cite{OzawaRieffel}) and groups of polynomial growth (see \cite{ChristRieffel}).
The proof in \cite{Rieffel02} strongly relies on the study of Gromov's
horofunction compactification (or metric compactification) of free
Abelian groups and fixed points under the corresponding group action;
this study was extended to finitely generated nilpotent groups in
\cite{Walsh07}. For a given discrete group $G$ endowed with a proper length
function $\ell$, the continuous functions on the corresponding horofunction
compactification can be viewed as a C$^{\ast}$-subalgebra of $\ell^{\infty}(G)$.
\\

The objective of the present article is to approach the question above,
mostly in the setting of virtually Abelian groups. Our approach is
inspired by those in \cite{Rieffel02}, \cite{OzawaRieffel}, and
employs metric geometry results on the approximation of length functions
by their stable semi-norms (see \cite{Burago}, \cite{LOZ21}). However,
compared to the group C$^{\ast}$-algebraic setting, the more complicated
crossed product setup causes increased technical difficulties. As
our main tool, we introduce the notion of groups that are separated
with respect to length functions: we say that the pair $(G,\ell)$
is \emph{separated} if the space of restrictions of the invariant means on
$\ell^{\infty}(G)$ to the continuous functions on the horofunction
compactification is in a certain sense sufficiently rich; for the
precise definition, see Definition \ref{SeparatedDefinition}. With this notion at hand, we
prove (among other things) the following theorem.

\begin{theorem*}[see Theorem \ref{MainTheorem} and Theorem \ref{MainTheorem2}]
Let $(\mathcal{A},\mathcal{H}_{A},D_{A})$ be a non-degenerate odd
(resp. even) spectral triple on a separable unital C$^{\ast}$-algebra
$A$ and assume that the induced Lipschitz semi-norm $L_{D_{A}}(a):=\left\Vert [D_{A},a]\right\Vert ,a\in\mathcal{A}$
defines a compact quantum metric space $(A,L_{D_{A}})$. Let further
$\alpha:G\rightarrow\text{Aut}(A)$ be a metrically equicontinuous
action of a finitely generated discrete group $G$, equipped with
a proper length function $\ell:G\rightarrow\mathbb{R}_{+}$, and assume
that there exists a finite index subgroup $H$ of $G$ that is separated
with respect to the restriction $\ell|_{H}$ and whose commutator
subgroup $[H,H]$ is finite. Then the even (resp. odd) spectral triple
defined in \cite{HSWZ13} is a spectral metric space. \end{theorem*}

As a consequence, we deduce the following statement on virtually Abelian
groups. It can be formulated in a more general way by replacing word
length functions with suitable orbit distance length functions; see Corollary
\ref{OrbitMetric}.

\begin{corollary*}[see Corollary \ref{VirtuallyAbelian}]
Let $(\mathcal{A},\mathcal{H}_{A},D_{A})$ be a non-degenerate odd
(resp. even) spectral triple on a separable unital C$^{\ast}$-algebra
$A$ and assume that the induced Lipschitz semi-norm $L_{D_{A}}(a):=\left\Vert [D_{A},a]\right\Vert ,a\in\mathcal{A}$
defines a compact quantum metric space $(A,L_{D_{A}})$. Let further
$\alpha:G\rightarrow\text{Aut}(A)$ be a metrically equicontinuous
action of a discrete virtually Abelian group $G$ that is finitely generated by a set $S$ with $S=S^{-1}$ and let $\ell:G\rightarrow\mathbb{R}_{+}$ be the corresponding
word length function. Then the even (resp. odd) spectral triple defined
in \cite{HSWZ13} is a spectral metric space. \end{corollary*}

The statements above can be applied to several natural examples,
some of which already occur in \cite{HSWZ13}. By using a result by
Christensen and Ivan on the construction of spectral triples on AF-algebras
(see \cite{CI06}), we can equip generalized Bunce-Deddens algebras
(as introduced in \cite{Orfanos10} and \cite{Carrion11}) associated with virtually Abelian groups
with compact quantum metric space structures. More generally, this
procedure works for all crossed products associated with suitable
actions of virtually Abelian groups on AF-algebras. Another family
of examples arises from higher-dimensional non-commutative tori; see \cite{Rieffel81} and
\cite{Rieffel90}. Any such C$^{\ast}$-algebra identifies with an iterated
crossed product by actions of the integers $\mathbb{Z}$. In particular, a
repeated application of the corollary above leads to spectral metric spaces. \\

\noindent \emph{Structure.} The paper is organized as follows. In Section 1, we recall the basic
notions of spectral triples, compact quantum metric spaces, and horofunction
compactifications. In the second one, we explain the construction of
odd and even spectral triples on crossed product C$^{\ast}$-algebras
by Hawkins, Skalski, White, and Zacharias, we introduce the notion of
groups that are separated with respect to length functions, and we prove
the main result of this article. Section 3 is concerned with the
study of length functions on free Abelian groups with respect to which
these groups are separated. We further discuss some implications of
Walsh's results in \cite{Walsh07}. In the last section, we consider
natural examples of C$^{\ast}$-algebras to which the statements of
the earlier sections can be applied. This selection includes generalized
Bunce-Deddens algebras and higher-dimensional non-commutative tori. \\

\vspace{3mm}


\noindent \textbf{Acknowledgements.} I am grateful to Adam Skalski and Piotr Nowak for bringing the questions studied in this article to my attention. They further contributed by providing fruitful discussions and by giving feedback on an earlier draft of this paper. I also wish to thank IMPAN where part of this work was carried out during a research visit.

\vspace{5mm}


\section{Preliminaries}

\vspace{3mm}

\subsection{General notation}

We will write $\mathbb{N}:=\left\{ 0,1,2,...\right\} $ and $\mathbb{N}_{\geq1}:=\left\{ 1,2,...\right\} $
for the natural numbers. The neutral element of a group is always
denoted by $e$, and for a set $S$, we write $\#S$ for the number
of elements in $S$. Scalar products of Hilbert spaces are linear
in the first variable, and we denote the bounded operators on a Hilbert
space $\mathcal{H}$ by $\mathcal{B}(\mathcal{H})$. Further, all
Hilbert spaces and C$^{\ast}$-algebras in this article are assumed
to be separable. We write $\otimes$ for the spatial tensor product
of C$^{\ast}$-algebras as well as for tensor products of Hilbert
spaces. For a discrete group $G$, we denote by $\ell^{2}(G)$ the
Hilbert space of all square summable functions $G\rightarrow\mathbb{C}$
and by $(\delta_{g})_{g\in G}$ the canonical orthonormal basis of
$\ell^2(G)$.

\vspace{3mm}


\subsection{Spectral triples and compact quantum metric spaces}

One of the key concepts in the theory of non-commutative geometry
is that of spectral triples introduced by Connes in \cite{Connes89}. \\

\begin{definition} Let $A$ be a separable unital C$^{\ast}$-algebra. 
\begin{enumerate}
\item An \emph{odd spectral triple} $(\mathcal{A},\mathcal{H},D)$ on $A$
consists of a $\ast$-representation $\pi:A\rightarrow\mathcal{B}(\mathcal{H})$,
a norm dense unital $\ast$-subalgebra $\mathcal{A}$ of $A$ and a densely
defined self-adjoint operator $D$ on $\mathcal{H}$ such that $(1+D^{2})^{-\frac{1}{2}}$
is compact and such that for every $a\in\mathcal{A}$, the domain
of $D$ is invariant under $\pi(a)$ and the commutator $[D,\pi(a)]$
is bounded. 
\item An \emph{even spectral triple} on $A$ consists of a triple $(\mathcal{A},\mathcal{H},D)$
as before and a $\mathbb{Z}_{2}$-grading (i.e., a Hilbert space decomposition
$\mathcal{H}=\mathcal{H}_{1}\oplus\mathcal{H}_{2}$ for which $\pi$
and $D$ decompose via $\pi=\pi_{1}\oplus\pi_{2}$ and 
\[
D=\left(\begin{array}{cc}
0 & D_{1}\\
D_{1}^{\ast} & 0
\end{array}\right)
\]
for suitable $D_1$).
\end{enumerate}
The operator $D$ from above is often called the triple's \emph{Dirac
operator}. \end{definition}

Following Connes \cite{Connes89}, given a spectral triple $(\mathcal{A},\mathcal{H},D)$,
one can define a \emph{Lipschitz semi-norm $L_{D}$ } on $\mathcal{A}$
via 
\[
L_{D}(a):=\left\Vert [D,\pi(a)]\right\Vert ,
\]
meaning that $L_{D}:\mathcal{A}\rightarrow\mathbb{R}_{+}$ is a semi-norm
whose domain is a dense subspace of $A$ that contains $1$ and for
which $L_{D}(1)=0$. (Note that there are various versions
of this concept; here, we follow the conventions in \cite{HSWZ13}.)
By \cite[Proposition 3.7]{Rieffel99} the semi-norm $L_{D}$ is \emph{lower semi-continuous}; that is, for
every $r>0$, the set $\{a\in\mathcal{A}\mid L_{D}(a)\leq r\}$ is
closed in $\mathcal{A}$ with respect to the subspace topology. $L_{D}$
further induces a pseudo-metric $d_{L_{D}}:\mathcal{S}(A)\rightarrow[0,\infty]$
on the state space $\mathcal{S}(A)$ of $A$ via 
\[
d_{L_{D}}(\psi,\psi^{\prime}):=\sup_{a\in \mathcal{A}:L_D(a)\leq1}|\psi(a)-\psi^{\prime}(a)|.
\]
Note that $d_{L_D}$ may take value $+ \infty$.

It is a natural question to ask when the topology on $\mathcal{S}(A)$
coming from $d_{L_{D}}$ coincides with the weak-$\ast$ topology
(see \cite{Rieffel98}, \cite{Rieffel99}). This is the defining property
of a compact quantum metric space. One necessary condition for this to happen is that the triple $(\mathcal{A},\mathcal{H},D)$ is \emph{non-degenerate} in the sense that the representation of $\mathcal{A}$
on $\mathcal{H}$ is faithful and $[D,\pi(a)]=0$ if and only if $a\in\mathbb{C}1$.
If the representation is faithful, we usually suppress it in the notation
and view $\mathcal{A}$ and $A$ as $\ast$-subalgebras of $\mathcal{B}(\mathcal{H})$.

\begin{definition}[{\cite[Definition 5.1]{Rieffel99} and \cite[Definition 2.2]{Rieffel04}}] Let $(\mathcal{A},\mathcal{H},D)$
be a non-degenerate spectral triple and define $L_{D}$ and $d_{L_{D}}$
as before. If the topology on $\mathcal{S}(A)$ induced by the metric
$d_{L_{D}}$ coincides with the weak-$\ast$ topology, $L_{D}$ is
called a \emph{Lip-norm}. In this case, we also say that the pair $(A,L_{D})$
is a \emph{compact quantum metric space} and that $(\mathcal{A},\mathcal{H},D)$
is a \emph{spectral metric space} (or also a \emph{metric spectral triple}). \end{definition}

Rieffel proved the following characterizations.

\begin{theorem}[{\cite[Theorem 1.8]{Rieffel98}}] \label{Characterization}
Let $(\mathcal{A},\mathcal{H},D)$ be a non-degenerate spectral triple
on a C$^{\ast}$-algebra $A$ and define $L_{D}$ and $d_{L_D}$ as before. Then
the following statements are equivalent:
\begin{enumerate}
\item The pair $(A,L_{D})$ defines a compact quantum metric space;
\item The image of $\{a\in\mathcal{A}\mid L_{D}(a)\leq1\}$ is totally bounded
in the quotient space $A/\mathbb{C}1$;
\item $d_{L_D}$ is bounded, and the set $\{a\in\mathcal{A}\mid L_{D}(a)\leq1\text{ and }\left\Vert a\right\Vert \leq1\}$
is totally bounded in $A$.
\end{enumerate}
\end{theorem}

\vspace{3mm}


\subsection{Horofunction compactifications\label{subsec:Horofunction-compactifications}}

In \cite{Rieffel02}, Rieffel demonstrated that
(twisted) group C$^{\ast}$-algebras of Abelian free groups $\mathbb{Z}^{m}$, $m \in \mathbb{N}$
equipped with the natural Dirac operators coming from word length
functions and restrictions of norms on Euclidean spaces, induce compact quantum metric spaces. His proof relies on
the study of Gromov's horofunction compactification (or metric compactification)
of these groups and fixed points under the corresponding group action.
Unfortunately, the approach does not cover other natural examples
such as reduced group C$^{\ast}$-algebras of word hyperbolic groups.
Only later, this class of C$^{\ast}$-algebras (and more generally
a class of certain filtered C$^{\ast}$-algebras) was treated by Ozawa
and Rieffel in \cite{OzawaRieffel} by employing their notion of Haagerup-type
condition. The results in \cite{Rieffel02} were extended to general
nilpotent-by-finite groups by Christ and Rieffel in \cite{ChristRieffel}.\\

Going back to Gromov \cite{Gromov} (see also \cite{Rieffel02}),
the horofunction compactification of a metric space $(Y,d)$ is defined
as follows. Consider the space $C(Y)$ of continuous functions on
$Y$ equipped with the topology of uniform convergence on bounded
sets. For $y_{0}\in Y$, define $C(Y,y_{0}):=\{f\in C(Y)\mid f(y_{0})=0\}$.
Then $C(Y,y_{0})$ is homeomorphic to $C_{\ast}(Y):=C(Y)/\mathbb{C}1$
equipped with the quotient topology, so in particular, $C(Y,y_{0})$
is independent of $y_{0}\in Y$. One can define a continuous embedding of
the space $Y$ into $C(Y,y_{0})$ via $y\mapsto f_{y}(\, \cdot \,):=d(y, \, \cdot \,)-d(y,y_{0})$.
The corresponding closure of $Y$ in $C(Y,y_{0})$ is denoted by $\widehat{Y}$.
If $(Y,d)$ is \emph{proper} in the sense that every closed ball in
$Y$ is compact, $\widehat{Y}$ is a compact Hausdorff space which
is called the \emph{horofunction compactification} of $Y$. The action
of the isometry group of $Y$ extends to a continuous action on $\widehat{Y}$
by homeomorphism. The space $\partial Y:=\widehat{Y}\setminus Y$
equipped with the subspace topology is called the \emph{horofunction
boundary} of $Y$.\\

In \cite[Section 4]{Rieffel02}, it was shown that if $(Y,d)$ is a
complete locally compact metric space, $C(\widehat{Y})$ can be described
as the (commutative) unital C$^{\ast}$-subalgebra $\mathcal{G}(Y,d)$
of $C_{b}(Y)$ generated by $C_{0}(Y)$ and the functions $Y\rightarrow\mathbb{C},y\mapsto f_{y}(x)$
where $x\in Y$ (i.e., $\widehat{Y}$ is homeomorphic to the character
spectrum of $\mathcal{G}(Y,d)$). \\

An important notion in Rieffel's work is that of weakly geodesic rays.

\begin{definition}[{\cite[Definition 4.3]{Rieffel02}}] Let $(Y,d)$
be a complete locally compact metric space and let $T\subseteq\mathbb{R}_{+}$
be an unbounded subset that contains $0$. Consider a function $\gamma:T\rightarrow Y$. 
\begin{itemize}
\item $\gamma$ is called a \emph{geodesic ray} if $d(\gamma(s),\gamma(t))=\left|s-t\right|$
for all $s,t\in T$; 
\item $\gamma$ is called an \emph{almost geodesic ray} if for every $\varepsilon>0$,
there exists an integer $N$ such that for all $t\geq s\geq N$, 
\[
\left|d(\gamma(t),\gamma(s))+d(\gamma(s),\gamma(0))-t\right|<\varepsilon;
\]
\item $\gamma$ is called a \emph{weakly geodesic ray} if for every $y\in Y$
and $\varepsilon>0$, there exists an integer $N$ such that if $s,t\geq N$,
then 
\[
\left|d(\gamma(t),\gamma(0))-t\right|<\varepsilon\hspace*{1em}\text{and}\hspace*{1em}\left|d(\gamma(t),y)-d(\gamma(s),y)-(t-s)\right|<\varepsilon.
\]
\end{itemize}
\end{definition}

It can be shown that every almost geodesic ray is weakly geodesic.
Further, the following theorem holds.

\begin{theorem}[{\cite[Theorem 4.7]{Rieffel02}}] \label{Busemann}
Let $(Y,d)$ be a complete locally compact metric space and let $\gamma:T\rightarrow Y\subseteq\widehat{Y}$
be a weakly geodesic ray. Then for every $f\in\mathcal{G}(Y,d)$, the
limit $\lim_{t\rightarrow\infty}f(\gamma(t))$ exists and gives a
(unique) element in $\partial Y$ in the sense that 
\[
\chi_{\gamma}:\mathcal{G}(Y,d)\rightarrow\mathbb{C},\chi_{\gamma}(f):=\lim_{t\rightarrow\infty}f(\gamma(t))
\]
defines a character on $\mathcal{G}(Y,d)$ whose restriction to $C_{0}(Y)$
vanishes. If $Y$ is proper and if the topology of $(Y,d)$ has a
countable base, then every point in $\partial Y$ is determined as
above by a weakly geodesic ray. \end{theorem}

\begin{definition}[{\cite[Definition 4.8]{Rieffel02}}] Let $(Y,d)$
be a complete locally compact metric space. A point in $\partial Y$
induced by a weakly geodesic ray $\gamma$ as in Theorem \ref{Busemann}
is called a \emph{Busemann point}. \end{definition}

In this article, we will mostly be concerned with the following setup
that occurs in \cite{Connes89}. Let $G$ be a discrete group equipped
with a \emph{length function} $\ell:G\rightarrow\mathbb{R}_{+}$; that is, $\ell(gh)\leq\ell(g)+\ell(h)$ and $\ell(g^{-1})=\ell(g)$
for all $g,h\in G$, and $\ell(g)=0$ exactly if $g=e$. Note that
every such length function induces a natural metric $d_{\ell}$ on
$G$ via $d_{\ell}(g,h):=\ell(g^{-1}h)$. The space $(G,d_{\ell})$
is proper if $\ell$ is \emph{proper} in the sense that the set $\{g\in G\mid\ell(g)\leq r\}$
is finite for all $r>0$. We will write $\overline{G}^{\ell}$ for
the horofunction compactification of $(G,d_{\ell})$ and $\partial_{\ell}G$
for the corresponding boundary. The canonical action of $G$ on itself
via left multiplication extends to a continuous action $G\curvearrowright\overline{G}^{\ell}$
which again restricts to an action $G\curvearrowright\partial_{\ell}G$ on the boundary.

Prototypes of length functions on finitely generated groups are \emph{word
length functions}: for every discrete group $G$ finitely generated
by a set $S$ with $S=S^{-1}$, the expression $\ell_{S}(g):=\min\{n\mid g=s_{1}...s_{n}\text{ where }s_{1},...,s_{n}\in S\}$,
$g\in G$ defines a length function on $G$.

\vspace{3mm}


\section{Spectral triples on crossed product C$^{\ast}$-algebras}

\vspace{3mm}


\subsection{Crossed product C$^{\ast}$-algebras\label{subsec:Crossed-product-C-algebras}}

Let $\alpha:G\rightarrow\text{Aut}(A)$ be an action of a discrete
group $G$ on a separable unital C$^{\ast}$-algebra $A$ and let
$\ell:G\rightarrow\mathbb{R}_{+}$ be a proper length function on
$G$. We will often write $g.a:=\alpha_{g}(a)$, where $g\in G$, $a\in A$.
Assume that $(\mathcal{A},\mathcal{H}_{A},D_{A})$ is an odd spectral
triple on $A$ via a faithful representation $\pi$ of $A$ and consider
the canonical odd spectral triple $(\mathbb{C}[G],\ell^{2}(G),M_{\ell})$
on $C_{r}^{\ast}(G)$. Here, $M_{\ell}$ denotes the multiplication
operator given by $M_{\ell}\delta_{g}:=\ell(g)\delta_{g}$ for $g\in G$,
and $\mathbb{C}[G]\subseteq C_{r}^{\ast}(G)$ is the span of all left
regular representation operators.

Recall that the reduced crossed product C$^{\ast}$-algebra $A\rtimes_{\alpha,r}G$
is defined as the C$^{\ast}$-subalgebra of $\mathcal{B}(\mathcal{H}_{A}\otimes\ell^{2}(G))$
generated by the operators $\widetilde{\pi}(a)$, $a\in A$ and $\lambda_{g}$,
$g\in G$ with 
\[
\widetilde{\pi}(a)(\xi\otimes\delta_{h}):=\pi(h^{-1}.a)\xi\otimes\delta_{h}\qquad\text{and}\qquad\lambda_{g}(\xi\otimes\delta_{h}):=\xi\otimes\delta_{gh}
\]
for $\xi\in\mathcal{H}_{A}$, $h\in G$. This definition does (up to isomorphism)
not depend on the choice of the faithful representation $\pi$. The
C$^{\ast}$-algebra $A$ naturally embeds into $A\rtimes_{\alpha,r}G$
via $a\mapsto\widetilde{\pi}(a)$. We will therefore often view $A$
as a C$^{\ast}$-subalgebra of $A\rtimes_{\alpha,r}G$ and suppress
$\pi$ and $\widetilde{\pi}$ in the notation. Further, we can canonically
view the reduced group C$^\ast$-algebra $C_{r}^{\ast}(G)$ as a C$^{\ast}$-subalgebra of $A\rtimes_{\alpha,r}G$.

\begin{lemma} \label{Independence} The C$^{\ast}$-subalgebra of
$\mathcal{B}(\mathcal{H}_{A}\otimes\ell^{2}(G))$ generated by $A\rtimes_{\alpha,r}G$ and
$\mathbb{C}1\otimes\ell^{\infty}(G)$
does (up to isomorphism) not depend on the choice of the faithful
representation $\pi:A\hookrightarrow\mathcal{B}(\mathcal{H}_{A})$.
\end{lemma}

\begin{proof} The argument is standard; compare, for instance, with
the proof of \cite[Proposition 4.1.5]{BrownOzawa}. Let $\pi:A\hookrightarrow\mathcal{B}(\mathcal{H}_{A})$
and $\pi^{\prime}:A\hookrightarrow\mathcal{B}(\mathcal{H}_{A}^{\prime})$
be two faithful representations of $A$, define $\widetilde{\pi}$
and $\widetilde{\pi}^{\prime}$ as above, and consider 
\begin{eqnarray*}
B_{1} & := & C^{\ast}(\widetilde{\pi}(A)\cup\{\lambda_{g}\mid g\in G\}\cup\mathbb{C}1\otimes\ell^{\infty}(G))\subseteq\mathcal{B}(\mathcal{H}_{A}\otimes\ell^{2}(G)),\\
B_{2} & := & C^{\ast}(\widetilde{\pi}^{\prime}(A)\cup\{\lambda_{g}\mid g\in G\}\cup\mathbb{C}1\otimes\ell^{\infty}(G))\subseteq\mathcal{B}(\mathcal{H}_{A}^{\prime}\otimes\ell^{2}(G)).
\end{eqnarray*}
We have to show that $B_{1}\cong B_{2}$ via $\widetilde{\pi}(a)\mapsto\widetilde{\pi}^{\prime}(a)$,
$\lambda_{g}\mapsto\lambda_{g}$ and $1\otimes f\mapsto1\otimes f$
for $a\in A$, $g\in G$, $f\in\ell^{\infty}(G)$. For every finite
subset $F\subseteq G$, define $P_{F}\in\ell^{\infty}(G)$ to be the
orthogonal projection onto the closure of $\text{Span}\{\delta_{g}\mid g\in F\}\subseteq\ell^{2}(G)$.
It is then easy to see that for all finite sequences $(a_{g})_{g\in G}\subseteq A$,
$(f_{g})_{g\in G}\subseteq\ell^{\infty}(G)$ with $a_{g}=0$ and $f_{g}=0$
for almost all $g\in G$, 
\[
\Vert\sum_{g\in G}\widetilde{\pi}(a_{g})(1\otimes f_{g})\lambda_{g}\Vert=\sup_{F\subseteq G\text{ finite}}\Vert(1\otimes P_{F})(\sum_{g\in G}\widetilde{\pi}(a_{g})(1\otimes f_{g})\lambda_{g})(1\otimes P_{F})\Vert
\]
and 
\[
\Vert\sum_{g\in G}\widetilde{\pi}^{\prime}(a_{g})(1\otimes f_{g})\lambda_{g}\Vert=\sup_{F\subseteq G\text{ finite}}\Vert(1\otimes P_{F})(\sum_{g\in G}\widetilde{\pi}^{\prime}(a_{g})(1\otimes f_{g})\lambda_{g})(1\otimes P_{F})\Vert.
\]
Now, for every finite subset $F\subseteq G$ and $a\in A$, 
\[
(1\otimes P_{F})\widetilde{\pi}(a)=(1\otimes P_{F})\widetilde{\pi}(a)(1\otimes P_{F})=\sum_{h\in F}\widetilde{\pi}(h^{-1}.a)\otimes e_{h,h},
\]
where $e_{g,h}$, $g,h\in F$ denote the canonical matrix units of
$P_{F}\mathcal{B}(\ell^{2}(G))P_{F}\cong M_{\#F}(\mathbb{C})$. This
implies that 
\begin{eqnarray*}
(1\otimes P_{F})(\sum_{g\in G}\widetilde{\pi}(a_{g})(1\otimes f_{g})\lambda_{g})(1\otimes P_{F}) & = & \sum_{g\in G}\sum_{h\in F\cap gF}\widetilde{\pi}(h^{-1}.a_{g})\otimes(f_{g}e_{h,g^{-1}h})\\
 & \in & \widetilde{\pi}(A)\otimes M_{\#F}(\mathbb{C})
\end{eqnarray*}
and similarly for $\widetilde{\pi}^{\prime}$. But $\widetilde{\pi}(A)\otimes M_{\#F}(\mathbb{C})\cong\widetilde{\pi}^{\prime}(A)\otimes M_{\#F}(\mathbb{C})$
canonically, and hence, the norms above coincide. \end{proof}

As in \cite[Section 2]{HSWZ13}, define a Dirac operator $D$ on $\mathcal{H}\oplus\mathcal{H}$
with $\mathcal{H}:=\mathcal{H}_{A}\otimes\ell^{2}(G)$ via 
\[
D:=\left(\begin{array}{cc}
0 & D_{A}\otimes1-i\otimes M_{\ell}\\
D_{A}\otimes1+i\otimes M_{\ell} & 0
\end{array}\right).
\]
and write 
\[
C_{c}(G,\mathcal{A}):=\left\{ \sum_{g\in G}a_{g}\lambda_{g}\mid(a_{g})_{g\in G}\subseteq\mathcal{A}\text{ with }a_{g}=0\text{ for almost all }g\in G\right\} .
\]
Then $C_{c}(G,\mathcal{A})$ is a dense $\ast$-subalgebra of $A\rtimes_{\alpha,r}G\subseteq\mathcal{B}(\mathcal{H})$.
For $x=\sum_{g\in G}a_{g}\lambda_{g}\in C_{c}(G,\mathcal{A})$ with
$(a_{g})_{g\in G}\subseteq\mathcal{A}$, we call $\text{supp}(x):=\{g\in G\mid a_{g}\neq0\}$
the \emph{support of $x$}. It was argued in \cite[Theorem 2.7]{HSWZ13}
that, under the assumption that $\mathcal{A}$ is invariant under
the action of $G$ and that $\sup_{g\in G}\left\Vert [D_{A},g.a]\right\Vert <\infty$
for every $a\in\mathcal{A}$, the triple $(C_{c}(G,\mathcal{A}),\mathcal{H}\oplus\mathcal{H},D)$
defines an even spectral triple on $A\rtimes_{\alpha,r}G$. Further,
if $(\mathcal{A},\mathcal{H}_{A},D_{A})$ is non-degenerate, so is
$(C_{c}(G,\mathcal{A}),\mathcal{H}\oplus\mathcal{H},D)$. (Note that
in \cite{HSWZ13}, the slightly different setup of proper translation
bounded integer-valued functions on $G$ is considered; however, the
results translate into our setting verbatim.)
Motivated by this, let us introduce the notion of metrically equicontinuous
actions.

\begin{definition}[{\cite[Definition 2.5]{HSWZ13}}] Let $(\mathcal{A},\mathcal{H}_{A},D_{A})$
be a non-degenerate odd spectral triple on a unital separable C$^{\ast}$-algebra $A$.
Assume that $L_{D_{A}}:\mathcal{A}\rightarrow[0,\infty)$, $L_{D_{A}}(a)=\left\Vert [D,a]\right\Vert $
is a Lipschitz seminorm such that the pair $(A,L_{D_{A}})$ is a compact
quantum metric space. An action $\alpha:G\rightarrow\text{Aut}(A)$
is called \emph{smooth} if $\alpha_{g}(\mathcal{A})\subseteq\mathcal{A}$
for every $g\in G$. If further $\sup_{g\in G}L_{D_{A}}(g.a)<\infty$
for every $a\in\mathcal{A}$, $\alpha$ is called \emph{metrically
equicontinuous}. \end{definition}

Recall that the horofunction compactification $\overline{G}^{\ell}$
of a discrete group $G$ equipped with a proper length function $\ell:G\rightarrow\mathbb{R}_{+}$
is the (compact) closure of the image of $G$ in $C(G,e)$ under the
embedding $g\mapsto f_{g}(\,\cdot\,):=d_{\ell}(g,\,\cdot\,)-d_{\ell}(g,e)$
and that the canonical action of $G$ on itself induces actions $\beta:G\curvearrowright C(\overline{G}^{\ell})$
and $G\curvearrowright C(\partial_{\ell}G)$. By the very construction,
for every $g\in G$, there exists a unique continuous bounded map $\varphi_{g}^{\ell}:\overline{G}^{\ell}\rightarrow\mathbb{C}$
defined by $\varphi_{g}^{\ell}(h):=\ell(h)-\ell(g^{-1}h)$ for $h\in G$.
These maps very naturally occur in our crossed product setting, as
the following lemma illustrates.

\begin{lemma} \label{Commutator identity} Let $\alpha:G\rightarrow\text{Aut}(A)$
be an action of a discrete group $G$ on a separable unital C$^{\ast}$-algebra
$A\subseteq\mathcal{B}(\mathcal{H}_{A})$ and let $\ell:G\rightarrow\mathbb{R}_{+}$
be a proper length function on $G$. For every $x=\sum_{g\in G}a_{g}\lambda_{g}\in C_{c}(G,A)\subseteq\mathcal{B}(\mathcal{H})$
with $(a_{g})_{g\in G}\subseteq A$, $\mathcal{H}:=\mathcal{H}_{A}\otimes\ell^{2}(G)$,
\[
[1\otimes M_{\ell},x]=\sum_{g\in G}(1\otimes\varphi_{g}^{\ell})a_{g}\lambda_{g}
\]
where $\varphi_{g}^{\ell}$ is viewed as a multiplication operator
$\delta_{h}\mapsto\varphi_{g}^{\ell}(h)$ in $\ell^{\infty}(G)\subseteq\mathcal{B}(\ell^{2}(G))$.
\end{lemma}

\begin{proof} One has that for every finite sum $x=\sum_{g\in G}a_{g}\lambda_{g}\in C_{c}(G,\mathcal{A})$
with $(a_{g})_{g\in G}\subseteq\mathcal{A}$ and $\xi\in\mathcal{H}$,
$h\in G$, 
\begin{eqnarray*}
 &  & [1\otimes M_{\ell},x](\xi\otimes\delta_{h})=(1\otimes M_{\ell})x(\xi\otimes\delta_{h})-x(1\otimes M_{\ell})(\xi\otimes\delta_{h})\\
 & = & \sum_{g\in G}\left(\ell(gh)-\ell(h)\right)(\alpha_{(gh)^{-1}}(a_{g})\xi\otimes\delta_{gh})=\sum_{g\in G}\varphi_{g}^{\ell}(gh)(\alpha_{(gh)^{-1}}(a_{g})\xi\otimes\delta_{gh})
\end{eqnarray*}
and hence, 
\[
[1\otimes M_{\ell},x]=\sum_{g\in G}(1\otimes\varphi_{g}^{\ell})a_{g}\lambda_{g}\in\mathcal{B}(\mathcal{H}),
\]
which implies the claim. \end{proof}

The maps $\varphi_{g}^{\ell}$, $g\in G$ further satisfy the following
1-cocycle condition which will become important in the later sections.

\begin{lemma} \label{1-cocycle} Let $G$ be a discrete group and
let $\ell:G\rightarrow\mathbb{R}_{+}$ be a proper length function
on $G$. Then $\varphi_{gh}^{\ell}=g.\varphi_{h}^{\ell}+\varphi_{g}^{\ell}$
for all $g,h\in G$. \end{lemma}

\begin{proof} For all $g,h,x\in G$, 
\begin{eqnarray*}
\varphi_{gh}^{\ell}(x) & = & \ell(x)-\ell(h^{-1}g^{-1}x)\\
 & = & \ell(x)-\ell(g^{-1}x)+\ell(g^{-1}x)-\ell(h^{-1}g^{-1}x)\\
 & = & \varphi_{g}^{\ell}(x)+\varphi_{h}^{\ell}(g^{-1}x).
\end{eqnarray*}
The claim then follows from the fact that the functions $\varphi_{gh}^{\ell}$,
$\varphi_{g}^{\ell}$, and $\varphi_{h}^{\ell}$ are continuous and
that $G$ is dense in $\overline{G}^{\ell}$. \end{proof}

By the discussion in Subsection \ref{subsec:Horofunction-compactifications},
the commutative C$^{\ast}$-algebra $C(\overline{G}^{\ell})$ is isomorphic
to the unital C$^{\ast}$-subalgebra $\mathcal{G}(G,\ell)$ of $\ell^{\infty}(G)$
generated by $C_{0}(G)$ and the set $\{\varphi_{g}^{\ell}\mid g\in G\}$,
where again the $\varphi_{g}^{\ell}$ are viewed as multiplication
operators. In the setting from before, define $\mathcal{C}(A,G,\ell)$
as the C$^{\ast}$-subalgebra of $\mathcal{B}(\mathcal{H})$ generated
by $A$, $\mathbb{C}1\otimes\mathcal{G}(G,\ell)$ and $C_{r}^{\ast}(G)$.
Then, $[1\otimes M_{l},x]\in\mathcal{C}(A,G,\ell)$ for every $x\in C_{c}(G,\mathcal{A})$.

For the sake of transparency, let us in the following denote the canonical
embeddings of $A$ and $\mathcal{G}(G,\ell)\cong C(\overline{G}^{\ell})$
by 
\begin{eqnarray}
\nonumber
\widetilde{\pi} &:& A\hookrightarrow\mathcal{C}(A,G,\ell)\subseteq\mathcal{B}(\mathcal{H}),\\
\nonumber
\pi_{\rtimes} &:& A\hookrightarrow A\rtimes_{\alpha,r}G\subseteq\mathcal{B}(\mathcal{H}),\\
\nonumber
\nu &:& \mathcal{G}(G,\ell)\hookrightarrow\mathcal{B}(\ell^{2}(G)),\\
\nonumber
\widetilde{\nu} &:& \mathcal{G}(G,\ell)\hookrightarrow\mathcal{C}(A,G,\ell)\subseteq\mathcal{B}(\mathcal{H}), \\
\nonumber
\nu_{\rtimes} &:& \mathcal{G}(G,\ell)\hookrightarrow C(\overline{G}^{\ell})\rtimes_{\beta,r}G.
\end{eqnarray}
Similarly, denote the left regular representation operators in $\mathcal{C}(A,G,\ell)\subseteq\mathcal{B}(\mathcal{H})$
by $\widetilde{\lambda}_{g}$, $g\in G$ and the ones in $\mathcal{B}(\ell^{2}(G))$,
$A\rtimes_{\alpha,r}G\subseteq\mathcal{B}(\mathcal{H})$ and $C(\overline{G}^{\ell})\rtimes_{\beta,r}G$
by $\lambda_{g}$, $g\in G$. Note that $\widetilde{\pi}(a)=\pi_{\rtimes}(a)$
and $\widetilde{\lambda}_{g}=\lambda_{g}$ in $\mathcal{B}(\mathcal{H})$
for all $a\in A$, $g\in G$.

\begin{proposition} \label{Isomorphism-1} The map $\mathcal{C}(A,G,\ell)\rightarrow(A\rtimes_{\alpha,r}G)\otimes(C(\overline{G}^{\ell})\rtimes_{\beta,r}G)\subseteq\mathcal{B}(\mathcal{H})\otimes\mathcal{B}(\ell^{2}(G)\otimes\ell^{2}(G))$
given by $\widetilde{\pi}(a)\mapsto\pi_{\rtimes}(a)\otimes1$, $\widetilde{\nu}(f)\mapsto1\otimes\nu_{\rtimes}(f)$
and $\widetilde{\lambda}_{g}\mapsto\lambda_{g}\otimes\lambda_{g}$
for $a\in A$, $f\in\mathcal{G}(G,\ell)\cong C(\overline{G}^{\ell})$
and $g\in G$ is a well-defined $\ast$-isomorphism onto its image.
\end{proposition}

\begin{proof} One can view $A$ as being covariantly and faithfully
represented on $\mathcal{H}=\mathcal{H}_{A}\otimes\ell^{2}(G)$ (via
$\widetilde{\pi}$ from before). In turn, by applying Lemma \ref{Independence},
we can both interpret $\mathcal{C}(A,G,\ell)$ as a C$^{\ast}$-subalgebra
of $\mathcal{B}(\mathcal{H}\otimes\ell^{2}(G))$ and as a C$^{\ast}$-subalgebra
of $\mathcal{B}(\mathcal{H})$. Write $\iota$ for the corresponding
embedding $\mathcal{C}(A,G,\ell)\hookrightarrow\mathcal{B}(\mathcal{H}\otimes\ell^{2}(G))$
and define a unitary $U:\mathcal{H}\otimes\ell^{2}(G)\rightarrow\mathcal{H}\otimes\ell^{2}(G)$
via $U(\xi\otimes\delta_{g}):=\widetilde{\lambda}_{g}\xi\otimes\delta_{g}$.
For $a\in A$, $\xi\in\mathcal{H}$, $g\in G$, one has
\begin{eqnarray*}
U(\iota\circ\widetilde{\pi})(a)U^{\ast}(\xi\otimes\delta_{g}) & = & U(\iota\circ\widetilde{\pi})(a)(\widetilde{\lambda}_{g^{-1}}\xi\otimes\delta_{g})\\
 & = & U(\alpha_{g^{-1}}(\widetilde{\pi}(a))\widetilde{\lambda}_{g^{-1}}\xi\otimes\delta_{g})\\
 & = & U(\widetilde{\lambda}_{g^{-1}}\widetilde{\pi}(a)\xi\otimes\delta_{g})\\
 & = & \widetilde{\pi}(a)\xi\otimes\delta_{g} \\
\nonumber
&=& \pi_{\rtimes}(a)\xi\otimes\delta_{g},
\end{eqnarray*}
so $U(\iota\circ\widetilde{\pi})(a)U^{\ast}=\pi_{\rtimes}(a)\otimes1$.
For $f\in C(\overline{G}^{\ell})\cong \mathcal{G}(G,\ell)$, $\xi\in\mathcal{H}$,
$g\in G$, 
\begin{eqnarray*}
U(\iota\circ\widetilde{\nu}(f))U^{\ast}(\xi\otimes\delta_{g}) & = & U(\iota\circ\widetilde{\nu})(f)(\widetilde{\lambda}_{g^{-1}}\xi\otimes\delta_{g})\\
 & = & f(g)U(\widetilde{\lambda}_{g^{-1}}\xi\otimes\delta_{g})\\
 & = & f(g)(\xi\otimes\delta_{g}),
\end{eqnarray*}
so $U(\iota\circ\widetilde{\nu})(f)U^{\ast}=1\otimes\nu(f)$. Lastly,
for $\xi\in\mathcal{H}$, $g,h\in G$, 
\begin{eqnarray*}
U\iota(\widetilde{\lambda}_{g})U^{\ast}(\xi\otimes\delta_{h}) & = & U\iota(\widetilde{\lambda}_{g})(\widetilde{\lambda}_{h^{-1}}\xi\otimes\delta_{h})\\
 & = & U(\widetilde{\lambda}_{h^{-1}}\xi\otimes\delta_{gh})\\
 & = & \widetilde{\lambda}_{g}\xi\otimes\delta_{gh} \\
\nonumber
&=& \lambda_{g}\xi\otimes\delta_{gh},
\end{eqnarray*}
so $U\iota(\widetilde{\lambda}_{g})U^{\ast}=\lambda_{g}\otimes\lambda_{g}$.
This implies that conjugation by $U$ implements a $\ast$-embedding of $\mathcal{C}(A,G,\ell)$ into  $(A\rtimes_{\alpha,r}G)\otimes C_{u}^{\ast}(G)$
via $\widetilde{\pi}(a)\mapsto\pi_{\rtimes}(a)\otimes1$, $\widetilde{\nu}(f)\mapsto1\otimes\nu(f)$
and $\widetilde{\lambda}_{g}\mapsto\lambda_{g}\otimes\lambda_{g}$
for $a\in A$, $f\in C(\overline{G}^{\ell})\cong\mathcal{G}(G,\ell)$
and $g\in G$. Here, $C_{u}^{\ast}(G)\subseteq\mathcal{B}(\ell^{2}(G))$
denotes the uniform Roe algebra which is generated by $\ell^{\infty}(G)$
and $C_{r}^{\ast}(G)$ in $\mathcal{B}(\ell^{2}(G))$. By \cite[Proposition 5.1.3]{BrownOzawa},
$C_{u}^{\ast}(G)\cong\ell^{\infty}(G)\rtimes_{r}G$ canonically where
the crossed product is taken with respect to the left translation
action. In particular, the C$^{\ast}$-subalgebra of $\mathcal{B}(\ell^{2}(G))$
generated by $\mathcal{G}(G,\ell)$ and $C_{r}^{\ast}(G)$ identifies
with $C(\overline{G}^{\ell})\rtimes_{\beta,r}G$. We deduce the claim.
\end{proof}

Note that the proof of Proposition \ref{Isomorphism-1} does not require
the action $\beta$ to be amenable.

For notational convenience, if $S$ is a subset of $G$, we define
\[
C_{c}(S,\mathcal{A}):=\left\{ \sum_{g\in S}a_{g}\lambda_{g}\mid(a_{g})_{g\in G}\subseteq\mathcal{A}\text{ with }a_{g}=0\text{ for almost all }g\in S\right\} \subseteq A\rtimes_{\alpha,r}G
\]
and $C_{c}(S,A)\subseteq A\rtimes_{\alpha,r}G$ analogously. If $S$
is a subgroup of $G$, then these spaces will be $\ast$-subalgebras
of $A\rtimes_{\alpha,r}G$.

\begin{lemma} \label{ConditionalExpectation-1} Let $H\subseteq G$
be a subgroup. Then there exists a contractive linear map $\mathbb{E}_{H}$
on $\mathcal{B}(\mathcal{H})$ such that for every $x=\sum_{g\in G}a_{g}\lambda_{g}\in C_{c}(G,\mathcal{A})\subseteq\mathcal{B}(\mathcal{H})$
with $(a_{g})_{g\in G}\subseteq\mathcal{A}$, the identities
$\mathbb{E}_{H}(x)=\sum_{h\in H}a_{h}\lambda_{h} \in C_c(H,\mathcal{A})$, $\mathbb{E}_{H}([D_{A}\otimes1,x]\lambda_{g^{-1}})\lambda_{g}=[D_{A}\otimes1,\mathbb{E}_{H}(x\lambda_{g^{-1}})\lambda_{g}]$
and $\mathbb{E}_{H}([1\otimes M_{\ell},x]\lambda_{g^{-1}})\lambda_{g}=[1\otimes M_{\ell},\mathbb{E}_{H}(x\lambda_{g^{-1}})\lambda_{g}]$
hold. \end{lemma}

\begin{proof} Let $(g_{i})_{i\in I}\subseteq G$ be a family of elements
with $G=\bigcup_{i\in I}Hg_{i}$ and $Hg_{i}\neq Hg_{j}$ for $i\neq j$.
For $i\in I$, write $P_{Hg_{i}}$ for the orthogonal projection onto
the closed subspace of $\ell^{2}(G)$ spanned by all orthonormal basis
vectors $\delta_{hg_{i}}$, $h\in H$. We claim that the linear map
$\mathbb{E}_{H}$ given by $x\mapsto\sum_{i\in I}(1\otimes P_{Hg_{i}})x(1\otimes P_{Hg_{i}})$
satisfies the required conditions, where the sum converges in the strong
operator topology. Indeed, for every $x\in\mathcal{B}(\mathcal{H})$,
\[
\Vert\sum_{i\in I}(1\otimes P_{Hg_{i}})x(1\otimes P_{Hg_{i}})\Vert\leq\sup_{i\in I}\left\Vert (1\otimes P_{Hg_{i}})x(1\otimes P_{Hg_{i}})\right\Vert \leq\left\Vert x\right\Vert ,
\]
as the operators $(1\otimes P_{Hg_{i}})x(1\otimes P_{Hg_{i}})$, $i\in I$
have pairwise orthogonal support and ranges. For $x=\sum_{g\in G}a_{g}\lambda_{g}\in C_{c}(G,\mathcal{A})\subseteq\mathcal{B}(\mathcal{H})$
with $(a_{g})_{g\in G}\subseteq\mathcal{A}$ and $\xi\in\mathcal{H}_{A}$,
$h^{\prime}\in H$, $i\in I$, we further find 
\begin{eqnarray}
\nonumber
\left(\mathbb{E}_{H}(x)\right)(\xi\otimes\delta_{h^{\prime}g_{i}}) &=& ((1\otimes P_{Hg_{i}})\sum_{g\in G}a_{g}\lambda_{g})(\xi\otimes\delta_{h^{\prime}g_{i}}) \\
\nonumber
&=& \sum_{g\in G}(1\otimes P_{Hg_{i}})\left(((gh^{\prime}g_{i})^{-1}.a_{g})\xi\otimes\delta_{gh^{\prime}g_{i}}\right) \\
\nonumber
&=& \sum_{h\in H}\left(((hh^{\prime}g_{i})^{-1}.a_{h})\xi\otimes\delta_{hh^{\prime}g_{i}}\right) \\
\nonumber
&=& (\sum_{h\in H}a_{h}\lambda_{h})(\xi\otimes\delta_{h^{\prime}g_{i}}),
\end{eqnarray}
so that $\mathbb{E}_{H}(x)=\sum_{h\in H}a_{h}\lambda_{h}$. For $g,g^{\prime}\in G$,
there exist $i\in I$ and $h^{\prime}\in H$ such that $gg^{\prime}=h^{\prime}g_{i}$.
It follows that
\begin{eqnarray}
\nonumber
\left(\mathbb{E}_{H}([D_{A}\otimes1,x]\lambda_{g^{-1}})\lambda_{g}\right)(\xi\otimes\delta_{g^{\prime}}) &=& ((1\otimes P_{Hg_{i}})\sum_{g^{\prime\prime}\in G}[D_{A}\otimes1,a_{g^{\prime\prime}}]\lambda_{g^{\prime\prime}g^{-1}})(\xi\otimes\delta_{h^{\prime}g_{i}}) \\
\nonumber
&=& (\sum_{g^{\prime\prime}\in Hg}[D_{A}\otimes1,a_{g^{\prime\prime}}]\lambda_{g^{\prime\prime}g^{-1}})(\xi\otimes\delta_{h^{\prime}g_{i}}) \\
\nonumber
&=& (\sum_{g^{\prime\prime}\in Hg}[D_{A}\otimes1,a_{g^{\prime\prime}}]\lambda_{g^{\prime\prime}})(\xi\otimes\delta_{g^{\prime}}) \\
\nonumber
&=& \left([D_{A}\otimes1,\mathbb{E}_{H}(x\lambda_{g^{-1}})\lambda_{g}]\right)(\xi\otimes\delta_{g^{\prime}})
\end{eqnarray}
and with Lemma \ref{Commutator identity},
\begin{eqnarray}
\nonumber
\left(\mathbb{E}_{H}([1\otimes M_{\ell},x]\lambda_{g^{-1}})\lambda_{g}\right)(\xi\otimes\delta_{g^{\prime}}) &=& ((1\otimes P_{Hg_{i}})\sum_{g\in G}(1\otimes\varphi_{g^{\prime\prime}}^{\ell})a_{g^{\prime\prime}}\lambda_{g^{\prime\prime}g^{-1}})(\xi\otimes\delta_{h^{\prime}g_{i}}) \\
\nonumber
&=& (\sum_{g^{\prime\prime}\in Hg}(1\otimes\varphi_{g^{\prime\prime}}^{\ell})a_{g^{\prime\prime}}\lambda_{g^{\prime\prime}g^{-1}})(\xi\otimes\delta_{h^{\prime}g_{i}}) \\
\nonumber
&=& (\sum_{g^{\prime\prime}\in Hg}(1\otimes\varphi_{g^{\prime\prime}}^{\ell})a_{g^{\prime\prime}}\lambda_{g^{\prime\prime}})(\xi\otimes\delta_{g^{\prime}}) \\
\nonumber
&=& ([1\otimes M_{\ell},\mathbb{E}_{H}(x\lambda_{g^{-1}})\lambda_{g}])(\xi\otimes\delta_{g^{\prime}}).
\end{eqnarray}
We deduce that $\mathbb{E}_{H}([D_{A}\otimes1,x]\lambda_{g^{-1}})\lambda_{g}=[D_{A}\otimes1,\mathbb{E}_{H}(x\lambda_{g^{-1}})\lambda_{g}]$
and $\mathbb{E}_{H}([1\otimes M_{\ell},x]\lambda_{g^{-1}})\lambda_{g}=[1\otimes M_{\ell},\mathbb{E}_{H}(x\lambda_{g^{-1}})\lambda_{g}]$,
as claimed. \end{proof}

\vspace{3mm}


\subsection{Crossed product C$^{\ast}$-algebras as compact quantum metric spaces\label{MainPart}}

Consider the setting of Subsection \ref{subsec:Crossed-product-C-algebras}; that is, let $(\mathcal{A},\mathcal{H}_{A},D_{A})$ be a non-degenerate
odd spectral triple on a separable unital C$^{\ast}$-algebra $A$
and assume that the induced Lipschitz semi-norm $L_{D_{A}}(a):=\left\Vert [D_{A},a]\right\Vert ,a\in\mathcal{A}$
defines a compact quantum metric space $(A,L_{D_{A}})$. Let further
$\alpha:G\rightarrow\text{Aut}(A)$ be a metrically equicontinuous
action of a discrete group and $\ell:G\rightarrow\mathbb{R}_{+}$
a proper length function on $G$. It is natural to ask whether the
even spectral triple $(C_{c}(G,\mathcal{A}),\mathcal{H}\oplus\mathcal{H},D)$
defined in Subsection \ref{subsec:Crossed-product-C-algebras} induces
a Lip-metric on the state space of $A\rtimes_{\alpha,r}G$. This question
was formulated in \cite{HSWZ13} in the case of $G=\mathbb{Z}$ and
length functions induced by finite symmetric generating sets. The
discussion in \cite[Subsection 2.3]{HSWZ13} implies the following
convenient criterion. We include its proof for the convenience of the reader.

\begin{proposition}[{\cite[Subsection 2.3]{HSWZ13}}] \label{QuantumMetricSpaceCharacterization}
The even spectral triple $(C_{c}(G,\mathcal{A}),\mathcal{H}\oplus\mathcal{H},D)$
defined in Subsection \ref{subsec:Crossed-product-C-algebras}
is a spectral metric space if and only if the set 
\begin{eqnarray}
\left\{ x\in C_{c}(G,\mathcal{A})\mid\left\Vert [D_{A}\otimes1,x]\right\Vert \leq1\text{ and }\left\Vert [1\otimes M_{\ell},x]\right\Vert \leq1\right\}  \label{set}
\end{eqnarray}
has totally bounded image in $(A\rtimes_{\alpha,r}G)/\mathbb{C}1$.
\end{proposition}

\begin{proof} By Theorem \ref{Characterization}, the even spectral
triple $(C_{c}(G,\mathcal{A}),\mathcal{H}\oplus\mathcal{H},D)$ is a spectral metric space if and only if the set $\mathcal{Q}:=\{x\in C_{c}(G,\mathcal{A})\mid\Vert[D,x\oplus x]\Vert\leq1\}$
has totally bounded image in the quotient space $(A\rtimes_{\alpha,r}G)/\mathbb{C}1$.
Denote the set in \eqref{set} by $\mathcal{Q}^{\prime}$. \\

``$\Rightarrow$'' Assume that the even spectral triple $(C_{c}(G,\mathcal{A}),\mathcal{H}\oplus\mathcal{H},D)$
is a spectral metric space. For $x=\sum_{g\in G}a_{g}\lambda_{g}\in\mathcal{Q}^{\prime}$
with $(a_{g})_{g\in G}\subseteq\mathcal{A}$, we have by
\begin{equation}
[D,x\oplus x]=\left(\begin{array}{cc}
0 & [D_{A}\otimes1,x]-i[1\otimes M_{\ell},x]\\{}
[D_{A}\otimes1,x]+i[1\otimes M_{\ell},x] & 0
\end{array}\right)\label{Commutator}
\end{equation}
that 
\[
\Vert[D,x\oplus x]\Vert\leq2\Vert[D_{A}\otimes1,x]\Vert+2\Vert[1\otimes M_{\ell},x]\Vert\leq4.
\]
This means that $\mathcal{Q}^{\prime}\subseteq4\mathcal{Q}$, and therefore,
the image of $\mathcal{Q}^{\prime}$ must be totally bounded in $(A\rtimes_{\alpha,r}G)/\mathbb{C}1$. \\

``$\Leftarrow$'' Assume that the image of $\mathcal{Q}^{\prime}$
is totally bounded in $(A\rtimes_{\alpha,r}G)/\mathbb{C}1$. From
\eqref{Commutator} it follows that $\Vert[D_{A}\otimes1,x]+i[1\otimes M_{\ell},x]\Vert\leq1$
and $\Vert[D_{A}\otimes1,x]-i[1\otimes M_{\ell},x]\Vert\leq1$ for
every $x\in\mathcal{Q}$. But then, $\Vert[D_{A}\otimes1,x]\Vert\leq2$
and $\Vert[1\otimes M_{\ell},x]\Vert\leq2$, and therefore $\mathcal{Q}\subseteq2\mathcal{Q}^{\prime}$.
We conclude that the image of $\mathcal{Q}$ in $(A\rtimes_{\alpha,r}G)/\mathbb{C}1$
must be totally bounded, and therefore, the triple $(C_{c}(G,\mathcal{A}),\mathcal{H}\oplus\mathcal{H},D)$
is a spectral metric space. \end{proof}

Proposition \ref{QuantumMetricSpaceCharacterization} implies that
in the treatment of the question above, it suffices to restrict to cosets of finite index subgroups.

\begin{lemma} \label{FiniteIndex} Let $(\mathcal{A},\mathcal{H}_{A},D_{A})$
be a non-degenerate odd spectral triple on a separable unital C$^{\ast}$-algebra
$A$ and assume that the induced Lipschitz semi-norm $L_{D_{A}}(a):=\left\Vert [D_{A},a]\right\Vert ,a\in\mathcal{A}$
defines a compact quantum metric space $(A,L_{D_{A}})$. Let further
$\alpha:G\rightarrow\text{Aut}(A)$ be a metrically equicontinuous
action of a finitely generated discrete group $G$ equipped with a
proper length function $\ell:G\rightarrow\mathbb{R}_{+}$ and let
$H\leq G$ be a finite index subgroup. Then the even spectral triple
$(C_{c}(G,\mathcal{A}),\mathcal{H}\oplus\mathcal{H},D)$ defined in
Subsection \ref{subsec:Crossed-product-C-algebras} is a spectral metric space
if and only if for every $g\in G$ the set of
all elements $x=\sum_{h\in H}a_{h}\lambda_{hg}\in C_{c}(Hg,\mathcal{A})$
with $(a_{h})_{h\in H}\subseteq\mathcal{A}$ satisfying $\left\Vert [D_{A}\otimes1,x]\right\Vert \leq1$
and $\left\Vert [1\otimes M_{\ell},x]\right\Vert \leq1$ has totally
bounded image in $(A\rtimes_{\alpha,r}G)/\mathbb{C}1$. \end{lemma}

\begin{proof} Set $\mathcal{Q}:=\left\{ x\in C_{c}(G,\mathcal{A})\mid\left\Vert [D_{A}\otimes1,x]\right\Vert \leq1\text{ and }\left\Vert [1\otimes M_{\ell},x]\right\Vert \leq1\right\} $,
and for $g\in G$, write $\mathcal{Q}_{g}$ for the set of all elements
$x=\sum_{h\in H}a_{h}\lambda_{hg}\in C_{c}(Hg,\mathcal{A})$ with
$(a_{h})_{h\in H}\subseteq\mathcal{A}$ satisfying $\left\Vert [D_{A}\otimes1,x]\right\Vert \leq1$
and $\left\Vert [1\otimes M_{\ell},x]\right\Vert \leq1$. \\

``$\Rightarrow$'' Assume that the triple $(C_{c}(G,\mathcal{A}),\mathcal{H}\oplus\mathcal{H},D)$
is a spectral metric space. By Proposition \ref{QuantumMetricSpaceCharacterization},
the set $\mathcal{Q}$ has totally bounded image in $(A\rtimes_{\alpha,r}G)/\mathbb{C}1$.
But $\mathcal{Q}_{g}$ is contained in $\mathcal{Q}$. It follows
that $\mathcal{Q}_{g}$ must also have totally bounded image in $(A\rtimes_{\alpha,r}G)/\mathbb{C}1$.
\\

``$\Leftarrow$'' Assume that $\mathcal{Q}_{g}$ has totally bounded
image in $(A\rtimes_{\alpha,r}G)/\mathbb{C}1$ for every $g\in G$
and let $g_{1},...,g_{m}\in G$ be elements with $G=\bigcup_{i=1}^{m}Hg_{i}$
and $Hg_{i}\neq Hg_{j}$ for $i\neq j$. For $x=\sum_{g\in G}a_{g}\lambda_{g}\in\mathcal{Q}$
with $(a_{g})_{g\in G}\subseteq\mathcal{A}$ and $i=1,...,m$, set
$x_{i}:=\sum_{h\in H}a_{hg_{i}}\lambda_{hg_{i}}$. Then $x=x_{1}+...+x_{m}$,
\begin{eqnarray*}
\left\Vert [1\otimes M_{\ell},x_{i}]\right\Vert  & = & \Vert\sum_{h\in H}(1\otimes\varphi_{hg_{i}})a_{hg_{i}}\lambda_{hg_{i}}\Vert\\
&=& \Vert[1\otimes M_{\ell},\mathbb{E}_{H}(x\lambda_{g_{i}^{-1}})\lambda_{g_{i}}]\Vert \\
 & = & \Vert\mathbb{E}_{H}([1\otimes M_{\ell},x]\lambda_{g_{i}^{-1}})\lambda_{g_{i}}\Vert\\
 & \leq & \left\Vert [1\otimes M_{\ell},x]\right\Vert \\
 & \leq & 1,
\end{eqnarray*}
where $\mathbb{E}_{H}$ is the contractive linear map appearing in
Lemma \ref{ConditionalExpectation-1}, and similarly
\[
\left\Vert [D_{A}\otimes1,x_{i}]\right\Vert =\Vert\mathbb{E}_{H}([D_{A}\otimes1,x]\lambda_{g_{i}^{-1}})\lambda_{g_{i}}\Vert\leq\Vert[D_{A}\otimes1,x]\Vert\leq1.
\]
It follows that $\mathcal{Q}\subseteq\mathcal{Q}_{g_{1}}+...+\mathcal{Q}_{g_{m}}$,
and hence, since the $\mathcal{Q}_{g_{i}}$ are assumed to have totally
bounded image in $(A\rtimes_{\alpha,r}G)/\mathbb{C}1$, the set $\mathcal{Q}$
also has totally bounded image in $(A\rtimes_{\alpha,r}G)/\mathbb{C}1$.
With Proposition \ref{QuantumMetricSpaceCharacterization}, we deduce
that the triple $(C_{c}(G,\mathcal{A}),\mathcal{H}\oplus\mathcal{H},D)$
is a spectral metric space. \end{proof}

For a group $G$, we denote its (normal) \emph{commutator subgroup}
(or \emph{derived subgroup}) generated by all commutators $[g,h]:=g^{-1}h^{-1}gh$,
$g,h\in G$ by $[G,G]$. Its \emph{Abelianization} is the commutative
group $G/[G,G]$. If $G$ is finitely generated, then so is its Abelianization,
which can therefore be written as a direct product $T\times\mathbb{Z}^{m}$,
where $m\geq0$ is the rank of $G/[G,G]$ and where $T$ is its torsion
subgroup.

Recall that an \emph{invariant mean} of a discrete group $G$ is a
state on $\ell^{\infty}(G)\subseteq\mathcal{B}(\ell^{2}(G))$ that
is invariant under the canonical action of $G$. A group $G$ is \emph{amenable}
if it admits an invariant mean.

As it will be convenient in Section \ref{SeparatedGroups}, we formulate
the following lemma as well as Definition \ref{SeparatedDefinition}
for pseudo-length functions instead of just length functions. A \emph{pseudo-length
function} on a discrete group $G$ is a map $\ell:G\rightarrow\mathbb{R}_{+}$
satisfying $\ell(gh)\leq\ell(g)+\ell(h)$, $\ell(g^{-1})=\ell(g)$
for all $g,h\in G$ and $\ell(e)=0$. As before, associate bounded
operators $\varphi_{g}^{\ell}\in\ell^{\infty}(G)\subseteq\mathcal{B}(\ell^{2}(G))$,
$g\in G$ with $\ell$ by defining $\varphi_{g}^{\ell}\delta_{h}:=(\ell(h)-\ell(g^{-1}h))\delta_{h}$
for $h\in G$. It is easy to check that the 1-cocycle identity in
Lemma \ref{1-cocycle} holds for (not necessarily proper) pseudo-length
functions as well.

\begin{lemma} \label{InducedHomomorphism} Let $G$ be a finitely
generated discrete group equipped with a pseudo-length function $\ell$.
Denote the projection onto the torsion-free component of the Abelianization
$G/[G,G]$ of $G$ by $p_{G}$. Then every invariant mean $\mu:\ell^{\infty}(G)\rightarrow\mathbb{C}$
induces a well-defined group homomorphism $\widehat{\mu}_{\ell}:\text{im}(p_{G})\rightarrow\mathbb{R}$
via $p_{G}(g)\mapsto\mu(\varphi_{g}^{\ell})$. \end{lemma}

\begin{proof} Note that $\varphi_{g}^{\ell}\in\ell^{\infty}(G)$
is self-adjoint for every $g\in G$. In combination with the 1-cocycle
identity, this implies that the map $G\rightarrow\mathbb{R}$, $g\mapsto\mu(\varphi_{g}^{\ell})$
is a well-defined group homomorphism. Every such homomorphism vanishes
on the commutator subgroup. The induced map on the Abelianization
must vanish on the torsion subgroup. This proves the statement. \end{proof}

The fundamental idea of our approach consists of showing that for
suitable groups and (pseudo-)length functions on them, the space of
all invariant means is sufficiently rich in the sense that it induces
many non-trivial group homomorphisms as in Lemma \ref{InducedHomomorphism}.
Let us therefore introduce the following notion.

\begin{definition} \label{SeparatedDefinition} Let $G$ be a finitely
generated discrete group equipped with a pseudo-length function $\ell$
and let $p_{G}$ be the projection onto the torsion-free component
of the Abelianization of $G$. We call $G$ \emph{separated with respect
to $\ell$} if 
\[
\text{Hom}(\text{im}(p_{G}),\mathbb{R})=\text{Span}\left\{ \widehat{\mu}_{\ell}\mid\mu\text{ invariant mean}\right\} .
\]
In this case, we also say that the pair $(G,\ell)$ is separated. \end{definition}

It is clear that every group that is separated with respect to a certain
length function has to be amenable.

For notational convenience, for subsets $S,T$ of a C$^{\ast}$-algebra
$A$ and $\varepsilon>0$, we write $S\subseteq_{\varepsilon}T$ if
for every $a\in S$, there exists $b\in T$ with $\left\Vert a-b\right\Vert <\varepsilon$.
For $\lambda>0$, we further denote the set of all elements $\lambda a$
with $a\in S$ by $\lambda S$.

\begin{proposition} \label{DiagonalApproximation} Let $G$ be a
finitely generated discrete group equipped with a proper length function
$\ell:G\rightarrow\mathbb{R}_{+}$. Assume that $G$ admits a finite
index subgroup $H$ that is separated with respect to the restriction
of $\ell$ to $H$ and let $A\subseteq\mathcal{B}(\mathcal{H}_{A})$
be a unital separable C$^{\ast}$-algebra on which $G$ acts. For
every $g\in G$, define 
\[
\mathcal{Q}_{1}^{g}:=\left\{ x\in C_{c}(Hg,A)\mid\left\Vert [1\otimes M_{\ell},x]\right\Vert \leq1\right\}.
\]
Then for every $\varepsilon>0$, there exists $\delta>0$ and finitely
many elements $g_{1},...,g_{n}\in G$ such that $\mathcal{Q}_{1}^{g}\subseteq_{\varepsilon}\delta\mathcal{Q}_{1}^{g}\cap C_{c}(K,A)$,
where $K:=\bigcup_{i=1}^{n}[H,H]g_{i}$.

Similarly, if $(\mathcal{A},\mathcal{H}_{A},D_{A})$ is a non-degenerate
spectral triple on $A$, $g\in G$ and 
\[
\mathcal{Q}_{2}^{g}:=\left\{ x\in C_{c}(Hg,\mathcal{A})\mid\left\Vert [1\otimes M_{\ell},x]\right\Vert \leq1,\left\Vert [D_{A}\otimes1,x]\right\Vert \leq1\right\} ,
\]
then for every $\varepsilon>0$, there exists $\delta>0$ and finitely
many elements $g_{1},...,g_{n}\in G$ such that $\mathcal{Q}_{2}^{g}\subseteq_{\varepsilon}\delta\mathcal{Q}_{2}^{g}\cap C_{c}(K,\mathcal{A})$,
where $K:=\bigcup_{i=1}^{n}[H,H]g_{i}$. \end{proposition}

Roughly speaking, Proposition \ref{DiagonalApproximation} states
that all elements $x\in C_{c}(Hg,A)$, $g\in G$ with $\left\Vert [1\otimes M_{\ell},x]\right\Vert \leq1$
(and $\left\Vert [D_{A}\otimes1,x]\right\Vert \leq1$) can suitably
be approximated by ones that are in some sense almost supported on
the commutator subgroup of $H$. This has important implications.
The proof of the proposition relies on the following variation of
the result in \cite[Section 2]{OzawaRieffel}.

\begin{lemma} \label{OzawaRieffel} Let $G$ be a finitely generated
discrete group equipped with a proper length function $\ell:G\rightarrow\mathbb{R}_{+}$,
let $\alpha:G\rightarrow\text{Aut}(A)$ an action of $G$ on a unital
separable C$^{\ast}$-algebra $A\subseteq\mathcal{B}(\mathcal{H}_{A})$,
and let $L\in\mathbb{R}$. For a non-trivial group homomorphism $\phi:G\rightarrow\mathbb{Z}$,
define an unbounded operator $M_{\phi}$ on $\ell^{2}(G)$ via $M_{\phi}\delta_{g}:=\phi(g)\delta_{g}$
for $g\in G$. Then for every $x=\sum_{g\in G}a_{g}\lambda_{g}\in C_{c}(G,A)$
with $(a_{g})_{g\in G}\subseteq A$, the operator $[1\otimes M_{\phi},x]$
has dense domain and is bounded. Further, 
\[
\Vert\sum_{g\in G:\left|\phi(g)\right|>N}a_{g}\lambda_{g}\Vert\leq\left(\sum_{k\in\mathbb{Z}:\left|k\right|>N}\frac{1}{(k+L)^{2}}\right)^{1/2}\left\Vert [1\otimes M_{\phi},x] + Lx \right\Vert 
\]
for every $N\in\mathbb{N}$ with $N\geq\left|L\right|$. \end{lemma}

\begin{proof} It is clear that $[1\otimes M_{\phi},x]$ has dense
domain and by the same computation as in the proof of Lemma \ref{Commutator identity},
\[
[1\otimes M_{\phi},x]=\sum_{g\in G}\phi(g)a_{g}\lambda_{g},
\]
so $[1\otimes M_{\phi},x]$ is bounded. To prove the inequality, define
a strong operator-continuous 1-parameter family $\mathbb{R}\rightarrow\mathcal{B}(\mathcal{H}_{A}\otimes\ell^{2}(G))$,
$t\mapsto U_{t}$ via $U_{t}(\xi\otimes\delta_{g}):=e^{it\phi(g)}(\xi\otimes\delta_{g})$
for $\xi\in\mathcal{H}_{A}$, $g\in G$. For fixed $N\in\mathbb{N}$
with $N\geq\left|L\right|$, we obtain a bounded linear map on $\mathcal{B}(\mathcal{H}_{A}\otimes\ell^{2}(G))$
via $\kappa(x)\eta:=(2\pi)^{-1}\int_{0}^{2\pi}f_{N}(t)U_{t}xU_{t}^{\ast}\eta dt$
for $\eta\in\mathcal{H}_{A}\otimes\ell^{2}(G)$ with the $L^{2}$-function
$f_{N}(t):=\sum_{k\in\mathbb{Z}:\left|k\right|>N}(k+L)^{-1}e^{-ikt}$
with prescribed Fourier coefficients $(k+L)^{-1}$ for $|k|>N$.
Then, for $x=\sum_{g\in G}a_{g}\lambda_{g}\in C_{c}(G,A)$ with $(a_{g})_{g\in G}\subseteq A$
and $\xi\in\mathcal{H}_{A}$, $h\in G$, 
\begin{eqnarray*}
\left[\kappa([1\otimes M_{\phi},x]+Lx)\right](\xi\otimes\delta_{h}) & = & \sum_{g\in G}\frac{\phi(g)+L}{2\pi}\left\{ \int_{0}^{2\pi}e^{it\phi(g)}f_{N}(t)dt\right\} \left(((gh)^{-1}.a_{g})\xi\otimes\delta_{gh}\right)\\
 & = & \sum_{g\in G:\left|\phi(g)\right|>N}\left(((gh)_{.}^{-1}a_{g})\xi\otimes\delta_{gh}\right).
\end{eqnarray*}
We get that $\kappa([1\otimes M_{\phi},x]+Lx)=\sum_{g\in G:\left|\phi(g)\right|>N}a_{g}\lambda_{g}$,
and hence, 
\begin{eqnarray*}
 &  & \Vert\sum_{g\in G:\left|\phi(g)\right|>N}a_{g}\lambda_{g}\Vert=\left\Vert \kappa([1\otimes M_{\phi},x]+Lx)\right\Vert \leq\frac{\left\Vert [1\otimes M_{\phi},x]+Lx\right\Vert }{2\pi}\int_{0}^{2\pi}\left|f_{N}(t)\right|dt\\
 & \leq & \frac{\left\Vert [1\otimes M_{\phi},x]+Lx\right\Vert }{\sqrt{2\pi}}\left(\int_{0}^{2\pi}\left|f_{N}(t)\right|^{2}dt\right)^{1/2}=\left(\sum_{k\in\mathbb{Z}:\left|k\right|>N}\frac{1}{(k+L)^{2}}\right)^{1/2}\left\Vert [1\otimes M_{\phi},x] + Lx \right\Vert .
\end{eqnarray*}
\end{proof}

We are now ready to prove Proposition \ref{DiagonalApproximation}.
As mentioned earlier, Rieffel's approach in \cite{Rieffel02} relies
on the construction of sufficiently many fixed points in the horofunction
boundaries of $\mathbb{Z}^{m}$, $m\in\mathbb{N}$. These fixed points
induce conditional expectations from the crossed product C$^{\ast}$-algebra
associated with the horofunction compactification onto the group C$^{\ast}$-algebra
$C_{r}^{\ast}(\mathbb{Z}^{m})$. Similarly, in the proof of Proposition
\ref{DiagonalApproximation}, we will make use of the assumption that
$(G,\ell)$ is separated to construct suitable maps onto the restricted
crossed product C$^{\ast}$-algebra $A\rtimes_{\alpha|_{H},r}H$.

\begin{proof}[{Proof of Proposition \ref{DiagonalApproximation}}]
We only prove the second statement of Proposition \ref{DiagonalApproximation}
since the first one follows similarly. So assume that $(\mathcal{A},\mathcal{H}_{A},D_{A})$
is a non-degenerate odd spectral triple on $A$, $g\in G$, pick $x=\sum_{h\in H}a_{h}\lambda_{hg}\in C_{c}(Hg,A)$
with $(a_{h})_{h\in H}\subseteq\mathcal{A}$, $\left\Vert [1\otimes M_{\ell},x]\right\Vert \leq1$,
$\left\Vert [D_{A}\otimes1,x]\right\Vert \leq1$, and fix $\varepsilon>0$.
As before, let $p_{H}:H\twoheadrightarrow\mathbb{Z}^{m}$ be the projection
onto the torsion-free component of the Abelianization of $H$ (i.e.,
$m$ is the rank of the finitely generated Abelian group $H/[H,H]$).
By our assumption, $H$ is separated with respect to the restriction
$\ell|_{H}$. We can therefore find linear combinations $\phi_{1},...,\phi_{m}$
of invariant means on $\ell^{\infty}(H)$ such that $\phi_{i}(\varphi_{h}^{\ell|_{H}})=p_{i}\circ p_{H}(h)$,
for every $h\in H$, $1\leq i\leq m$ where $p_{i}:\mathbb{Z}^{m}\rightarrow\mathbb{Z}$
is the projection onto the $i$-th component of $\mathbb{Z}^{m}$.
These functionals induce maps $\ell^{\infty}(G)\rtimes_{\beta|_{H},r}H\rightarrow C_{r}^{\ast}(H)$
via $f\lambda_{h}\mapsto\phi_{i}(f|_{H})\lambda_{h}$ for $f\in\ell^{\infty}(G)$,
$h\in H$ and composition with the isomorphism from Proposition \ref{Isomorphism-1},
and an application of Fell's absorption principle (see \cite[Proposition 4.1.7]{BrownOzawa})
leads to bounded maps $P_{i}:\mathcal{C}(A,H,\ell)\rightarrow A\rtimes_{\alpha|_H,r}H$
via $P_{i}(a(1\otimes f)\lambda_{h}):=\phi_{i}(f|_{H})a\lambda_{h}$
for $a\in A$, $f\in\mathcal{G}(G,\ell)$, $h\in H$. Here, $\mathcal{C}(A,H,\ell)$
is the C$^{\ast}$-subalgebra of $\mathcal{B}(\mathcal{H})$ with
$\mathcal{H}:=\mathcal{H}_{A}\otimes\ell^{2}(G)$ generated by $A$,
$\mathbb{C}1\otimes\mathcal{G}(G,\ell)$ and $C_{r}^{\ast}(H)$. For every $i$, write $\mathbb{E}_{i}$ for
the contractive linear map on $\mathcal{B}(\mathcal{H})$ associated with the subgroup $\text{ker}(p_i \circ p_H) \leq G$ as in Lemma \ref{ConditionalExpectation-1}.

We proceed inductively. Define $L:=\max\{\left\Vert P_{1}\right\Vert ,...,\left\Vert P_{m}\right\Vert \}$
and note that $\varphi_{h}^{\ell|_{H}}=\varphi_{h}^{\ell}|_{H}$ for
every $h\in H$. By applying $P_{1}$ to $[1\otimes M_{\ell},x]\lambda_{g^{-1}}\in\mathcal{C}(A,H,\ell)$
and by using the identity in Lemma \ref{Commutator identity}, we
find 
\begin{eqnarray}
\nonumber
\Vert\phi_{1}(\varphi_{g}^{\ell}|_{H})x\lambda_{g^{-1}}+[1\otimes M_{p_{1}\circ p_{H}},x\lambda_{g^{-1}}]\Vert &=&  \Vert\sum_{h\in H}\left\{ \phi_{1}(\varphi_{g}^{\ell}|_{H})+p_{1}\circ p_{H}(h)\right\} a_{h}\lambda_{h}\Vert \\
\nonumber
&=& \Vert\sum_{h\in H}\phi_{1}(\varphi_{hg}^{\ell}|_{H})a_{h}\lambda_{h}\Vert \\
\nonumber
&=&   \Vert P_{1}([1\otimes M_{\ell},x]\lambda_{g^{-1}})\lambda_{g}\Vert \\
\nonumber
&\leq&  L,
\end{eqnarray}
In combination with Lemma \ref{OzawaRieffel} (where the constant is taken to be $\phi_{1}(\varphi_{g}^{\ell}|_{H})$), this implies that there
exists $N_{1}\in\mathbb{N}$ (which is independent of $x\in\mathcal{Q}_{2}^{g}$)
with 
\[
\Vert\sum_{h\in H:\left|p_{1}\circ p_{H}(h)\right|>N_{1}}a_{h}\lambda_{hg}\Vert\leq\Vert\sum_{h\in H:\left|p_{1}\circ p_{H}(h)\right|>N_{1}}a_{h}\lambda_{h}\Vert\leq m^{-1}\varepsilon.
\]
For every $-N_{1}\leq i\leq N_{1}$, choose $h_{i}\in H$ with $p_{1}\circ p_{H}(h_{i})=i$
and define an element in the crossed product via $x_{1}:=\sum_{h\in H:\left|p_{1}\circ p_{H}(h)\right|\leq N_{1}}a_{h}\lambda_{hg}\in A\rtimes_{\alpha,r}G$.
Then $\left\Vert x-x_{1}\right\Vert \leq m^{-1}\varepsilon$, 
\begin{eqnarray*}
\Vert[1\otimes M_{\ell},x_{1}]\Vert & = & \Vert\sum_{h\in G:\left|p_{1}\circ p_{H}(h)\right|\leq N_{1}}(1\otimes\varphi_{hg}^{\ell})a_{h}\lambda_{hg}\Vert\\
 & = & \Vert\sum_{-N_{1}\leq i\leq N_{1}}\sum_{h\in\ker(p_{1}\circ p_{H})}(1\otimes\varphi_{hh_{i}g}^{\ell})a_{hh_{i}}\lambda_{hh_{i}g}\Vert\\
&=& \Vert\sum_{-N_{1}\leq i\leq N_{1}}\mathbb{E}_{1}([1\otimes M_{\ell},x\lambda_{(h_{i}g)^{-1}}]\lambda_{h_{i}g})\Vert \\
 & = & \Vert\sum_{-N_{1}\leq i\leq N_{1}}\mathbb{E}_{1}([1\otimes M_{\ell},x]\lambda_{(h_{i}g)^{-1}})\lambda_{h_{i}g}\Vert\\
 & \leq & 2N_{1}+1,
\end{eqnarray*}
and similarly,
\[
\Vert[D_{A}\otimes1,x_{1}]\Vert=\Vert\sum_{-N_{1}\leq i\leq N_{1}}\mathbb{E}_{1}([D_{A}\otimes1,x]\lambda_{(h_{i}g)^{-1}})\lambda_{h_{i}g}\Vert\leq2N_{1}+1.
\]
In the same way, we can now apply $P_{2}$ to $[1\otimes M_{\ell},x_{1}]\lambda_{g^{-1}}$
and invoke Lemma \ref{OzawaRieffel} again to find $N_{2}\in\mathbb{N}$
with $\left\Vert x_{1}-x_{2}\right\Vert \leq m^{-1}\varepsilon$,
$\left\Vert [1\otimes M_{\ell},x_{2}]\right\Vert \leq(2N_{1}+1)(2N_{2}+1)$
and $\left\Vert [D_{A}\otimes1,x_{1}]\right\Vert \leq(2N_{1}+1)(2N_{2}+1)$,
where $x_{2}:=\sum_{h\in H:\left|p_{1}\circ p_{H}(h)\right|\leq N_{1},\left|p_{2}\circ p_{H}(h)\right|\leq N_{2}}a_{h}\lambda_{hg}\in A\rtimes_{\alpha,r}G$.
Performing these steps repeatedly leads to a sequence of natural numbers
$N_{1},...,N_{m}\in\mathbb{N}$ and elements $x_{1},...,x_{m}\in A\rtimes_{\alpha,r}G$
given by 
\[
x_{i}:=\sum_{h\in H:\left|p_{1}\circ p_{H}(h)\right|\leq N_{1},...,\left|p_{i}\circ p_{H}(h)\right|\leq N_{i}}a_{h}\lambda_{hg}
\]
for which $\left\Vert x_{i}-x_{i+1}\right\Vert <m^{-1}\varepsilon$,
$\left\Vert [1\otimes M_{\ell},x_{i}]\right\Vert \leq(2N_{1}+1)...(2N_{i}+1)$
and $\left\Vert [D_{A}\otimes1,x_{i}]\right\Vert \leq(2N_{1}+1)...(2N_{i}+1)$.
For $i=m$, we in particular have 
\[
\left\Vert x-x_{m}\right\Vert \leq\left\Vert x-x_{1}\right\Vert +\left\Vert x_{1}-x_{2}\right\Vert +...+\left\Vert x_{m-1}-x_{m}\right\Vert <\varepsilon.
\]
Set $\widetilde{N}:=\max\{N_{1},...,N_{m}\}$ and note that $p_{H}(\text{supp}(x_{m}\lambda_{g^{-1}}))$
is contained in the $\widetilde{N}$-ball of $\mathbb{Z}^{m}$ with
respect to the restriction of the supremum norm on $\mathbb{R}^{m}$
to $\mathbb{Z}^{m}$. It follows that there exist elements $g_{1},...,g_{n}\in G$
with $\text{supp}(x_{m})\subseteq K:=\bigcup_{i=1}^{n}[H,H]g_{i}.$
We therefore get that $\mathcal{Q}_{2}^{g}\subseteq_{\varepsilon}(2N_{1}+1)...(2N_{m}+1)\mathcal{Q}_{2}^{g}\cap C_{c}(K,\mathcal{A})$,
which finishes the proof. \end{proof}

\begin{theorem} \label{MainTheorem} Let $(\mathcal{A},\mathcal{H}_{A},D_{A})$
be a non-degenerate odd spectral triple on a separable unital C$^{\ast}$-algebra
$A$ and assume that the induced Lipschitz semi-norm $L_{D_{A}}(a):=\left\Vert [D_{A},a]\right\Vert ,a\in\mathcal{A}$
defines a compact quantum metric space $(A,L_{D_{A}})$. Let further
$\alpha:G\rightarrow\text{Aut}(A)$ be a metrically equicontinuous
action of a finitely generated discrete group $G$ equipped with a
proper length function $\ell:G\rightarrow\mathbb{R}_{+}$ and assume
that there exists a finite index subgroup $H$ of $G$ that is separated
with respect to the restricted length function $\ell|_{H}$. As before,
define $\mathcal{H}:=\mathcal{H}_{A}\otimes\ell^{2}(G)$ and $D$
as in Subsection \ref{subsec:Crossed-product-C-algebras}. Then the
even spectral triple $(C_{c}(G,\mathcal{A}),\mathcal{H}\oplus\mathcal{H},D)$
is a spectral metric space if and only if for every $g\in G$,
the set of all elements $x=\sum_{h\in[H,H]}a_{h}\lambda_{hg}\in C_{c}(G,\mathcal{A})$
with $(a_{h})_{h\in[H,H]}\subseteq\mathcal{A}$ satisfying $\left\Vert [D_{A}\otimes1,x]\right\Vert \leq1$
and $\left\Vert [1\otimes M_{\ell},x]\right\Vert \leq1$ has totally
bounded image in $(A\rtimes_{\alpha,r}G)/\mathbb{C}1$.

In particular, if $[H,H]$ is finite, then $(C_{c}(G,\mathcal{A}),\mathcal{H}\oplus\mathcal{H},D)$
is a spectral metric space. \end{theorem}

\begin{proof} For $g\in G$ ,set $\mathcal{Q}_{g}:=\left\{ x\in C_{c}(Hg,\mathcal{A})\mid\left\Vert [D_{A}\otimes1,x]\right\Vert \leq1\text{ and }\left\Vert [1\otimes M_{\ell},x]\right\Vert \leq1\right\} $
and write $\mathcal{Q}_{g}^{\prime}$ for the set of all elements
$x=\sum_{h\in[H,H]}a_{h}\lambda_{hg}\in C_{c}(Hg,\mathcal{A})$ with
$(a_{h})_{h\in[H,H]}\subseteq\mathcal{A}$ satisfying $\left\Vert [D_{A}\otimes1,x]\right\Vert \leq1$
and $\left\Vert [1\otimes M_{\ell},x]\right\Vert \leq1$. The ``only
if'' direction follows in the same way as in the proof of Lemma \ref{FiniteIndex}.
For the ``if'' direction, assume that $\mathcal{Q}_{g}^{\prime}$
has totally bounded image in $(A\rtimes_{\alpha,r}G)/\mathbb{C}1$
for every $g\in G$ and let $\varepsilon>0$. By Proposition \ref{DiagonalApproximation},
for fixed $g\in G$, we find $\delta>0$ and finitely many elements
$g_{1},...,g_{n}\in G$ such that $\mathcal{Q}_{g}\subseteq_{\varepsilon/4}\delta\mathcal{Q}_{g}\cap C_{c}(K,\mathcal{A})$,
where $K:=\bigcup_{i=1}^{n}[H,H]g_{i}$. In other words, for every
$x\in\mathcal{Q}_{g}$, there exists $y\in\delta\mathcal{Q}_{g}$ of
the form $y=\sum_{i=1}^{n}\sum_{h\in[H,H]}b_{hg_{i}}\lambda_{hg_{i}}$
with $b_{hg_{i}}\in\mathcal{A}$ for $h\in[H,H]$, $i=1,...,n$ such
that $\left\Vert x-y\right\Vert <\frac{\varepsilon}{4}$. For every
$i$, set $y_{i}:=\sum_{h\in[H,H]}b_{hg_{i}}\lambda_{hg_{i}}$. By
the same argument as in the proof of Lemma \ref{FiniteIndex} and
Proposition \ref{DiagonalApproximation}, 
\begin{eqnarray*}
\left\Vert [1\otimes M_{\ell},y_{i}]\right\Vert  & = & \Vert\sum_{h\in[H,H]}(1\otimes\varphi_{hg_{i}})b_{hg_{i}}\lambda_{hg_{i}}\Vert\\
&=& \Vert\mathbb{E}_{[H,H]}([1\otimes M_{\ell},y\lambda_{g_{i}^{-1}}]\lambda_{g_{i}})\Vert \\
 & = & \Vert[1\otimes M_{\ell},\mathbb{E}_{[H,H]}(y\lambda_{g_{i}^{-1}})\lambda_{g_{i}}]\Vert \\
 & \leq & \left\Vert [1\otimes M_{\ell},y]\right\Vert \\
 & \leq & \delta
\end{eqnarray*}
and similarly,
\[
\left\Vert [D_{A}\otimes1,y_{i}]\right\Vert =\Vert\mathbb{E}_{[H,H]}([D_{A}\otimes1,y]\lambda_{g_{i}^{-1}})\lambda_{g_{i}}\Vert\leq\Vert[D_{A}\otimes1,y]\Vert\leq\delta,
\]
where $\mathbb{E}_{[H,H]}$ is the contractive linear map from Lemma
\ref{ConditionalExpectation-1}. We conclude that $\mathcal{Q}_{g}\subseteq_{\varepsilon/4}\mathcal{R}$,
where $\mathcal{R}:=\delta\mathcal{Q}_{g_{1}}^{\prime}+...+\delta\mathcal{Q}_{g_{n}}^{\prime}$.
From our assumption, it can easily be derived that $\mathcal{R}$ has
a totally bounded image in $(A\rtimes_{\alpha,r}G)/\mathbb{C}1$. We
hence find finitely many elements $x_{1},...,x_{m}\in\mathcal{R}$
such that for every $y\in\mathcal{R}$, there exists $1\leq i\leq m$
with $\left\Vert (y-x_{i})+\mathbb{C}1\right\Vert <\frac{\varepsilon}{4}$.
For every $i$, choose $\widetilde{x}_{i}\in\mathcal{Q}_{g}$ with
$\left\Vert (x_{i}-\widetilde{x}_{i})+\mathbb{C}1\right\Vert <\frac{\varepsilon}{2}$,
if possible. We claim that the $\varepsilon$-balls around the $\widetilde{x}_{i}+\mathbb{C}1$
cover the image of $\mathcal{Q}_{g}$ in $(A\rtimes_{\alpha,r}G)/\mathbb{C}1$.
Indeed, for $x\in\mathcal{Q}_{g}$, there exists $y\in\mathcal{R}$
with $\left\Vert x-y\right\Vert <\frac{\varepsilon}{4}$, and we find
$i$ with $\left\Vert (y-x_{i})+\mathbb{C}1\right\Vert <\frac{\varepsilon}{4}$.
By $\left\Vert (x-x_{i})+\mathbb{C}1\right\Vert \leq\left\Vert (x-y)+\mathbb{C}1\right\Vert +\left\Vert (y-x_{i})+\mathbb{C}1\right\Vert <\frac{\varepsilon}{2}$,
the element $\widetilde{x}_{i}\in\mathcal{Q}_{g}$ exists and

\[
\left\Vert (x-\widetilde{x}_{i})+\mathbb{C}1\right\Vert \leq\left\Vert (x-y)+\mathbb{C}1\right\Vert +\left\Vert (y-x_{i})+\mathbb{C}1\right\Vert +\left\Vert (x_{i}-\widetilde{x}_{i})+\mathbb{C}1\right\Vert <\varepsilon.
\]
The claim follows. Hence, the image of $\mathcal{Q}_{g}$ in $(A\rtimes_{\alpha,r}G)/\mathbb{C}1$
is totally bounded, and thus, by Lemma \ref{FiniteIndex}, the even spectral
triple $(C_{c}(G,\mathcal{A}),\mathcal{H}\oplus\mathcal{H},D)$ is a spectral metric space

For the proof of the second statement, assume that the derived subgroup
$[H,H]$ is finite and fix $g\in G$. We proceed by arguing along
the lines of the proof of \cite[Theorem 2.11]{HSWZ13}. For $x=\sum_{h\in[H,H]}a_{h}\lambda_{hg}\in\mathcal{Q}_{g}^{\prime}$
with $(a_{h})_{h\in[H,H]}\subseteq\mathcal{A}$, one has that for every
$\xi,\eta$ in the domain of $D_{A}$ and $h\in H$, 
\begin{eqnarray*}
\left\langle [D_{A},a_{h}]\xi,\eta\right\rangle  & = & \left\langle (D_{A}\otimes1)x\lambda_{(hg)^{-1}}(\xi\otimes\delta_{e}),\eta\otimes\delta_{e}\right\rangle -\left\langle x\lambda_{(hg)^{-1}}(D_{A}\otimes1)(\xi\otimes\delta_{e}),\eta\otimes\delta_{e}\right\rangle \\
 & = & \left\langle [D_{A}\otimes1,x](\xi\otimes\delta_{(hg)^{-1}}),\eta\otimes\delta_{e}\right\rangle 
\end{eqnarray*}
and therefore, $\left\Vert [D_{A},a_{h}]\right\Vert \leq\left\Vert [D_{A}\otimes1,x]\right\Vert \leq1$.
Similarly, for $h\in H$ and $\xi,\eta\in\mathcal{H}_{A}$, 
\[
\left\langle [1\otimes M_{\ell},x](\xi\otimes\delta_{(hg)^{-1}}),\eta\otimes\delta_{e}\right\rangle =\left\langle x(1\otimes M_{\ell})(\xi\otimes\delta_{(hg)^{-1}}),\eta\otimes\delta_{e}\right\rangle =\ell(hg)\left\langle a_{h}\xi,\eta\right\rangle 
\]
so that $\left\Vert a_{h}\right\Vert \leq(\ell(hg))^{-1}\left\Vert [1\otimes M_{\ell},x]\right\Vert \leq L$
for every $h\in[H,H]\setminus\{g^{-1}\}$, where
\[
L:=\max\{(\ell(hg))^{-1}\mid h\in[H,H]\setminus\{g^{-1}\}\}.
\]
It follows that $\mathcal{Q}_{g}^{\prime}$
is contained in the set of all $x=\sum_{h\in[H,H]}a_{h}\lambda_{hg}\in C_{c}(Hg,\mathcal{A})$
with $(a_{h})_{h\in[H,H]}\subseteq\mathcal{A}$ satisfying $\left\Vert [D_{A},a_{h}]\right\Vert \leq L^\prime$
for all $h\in H$ and $\left\Vert a_{h}\right\Vert \leq L^\prime$ for all
$h\in[H,H]\setminus\{g^{-1}\}$ with $L^{\prime}:=\max\{1,L\}$. Denote this set by $\mathcal{S}_{g}$.
We claim that $\mathcal{S}_{g}$ has totally bounded image in $(A\rtimes_{\alpha,r}G)/\mathbb{C}1$,
which then implies that $\mathcal{Q}_{g}^{\prime}$ has totally bounded
image in $(A\rtimes_{\alpha,r}G)/\mathbb{C}1$, and hence, by the previous
part, that the triple $(C_{c}(G,\mathcal{A}),\mathcal{H}\oplus\mathcal{H},D)$
is a spectral metric space. Indeed, by Theorem \ref{Characterization},
the set $F:=\{a\in\mathcal{A}\mid\left\Vert a\right\Vert \leq L^\prime \text{ and }\left\Vert [D_{A},a]\right\Vert \leq L^\prime\}$
is totally bounded in $A$. For every $\varepsilon>0$, we can hence
pick a finite subset $F_{1}$ of $F$ such that the $\frac{\varepsilon}{\#[H,H]}$-balls
around its elements cover $F$. Similarly, we can choose a finite subset
$F_{2}$ of $\{a\in\mathcal{A}\mid\left\Vert [D_{A},a]\right\Vert \leq L^\prime \}$
such that the $\frac{\varepsilon}{\#[H,H]}$-balls around the image
of the elements of $F_{2}$ in $A/\mathbb{C}1$ covers the image of
$\{a\in\mathcal{A}\mid\left\Vert [D_{A},a]\right\Vert \leq L^\prime\}$. From
this, we can deduce that if $g\in[H,H]$, the image of 
\[
\left\{ \sum_{h\in[H,H]}f_{h}\lambda_{hg}\mid f_{g^{-1}}\in F_{2}\text{ and }f_{h}\in F_{1}\text{ for }h\neq g^{-1}\right\} 
\]
is an $\varepsilon$-net for the image of $\mathcal{S}_{g}$ in $(A\rtimes_{\alpha,r}G)/\mathbb{C}1$,
and similarly, the image of 
\[
\left\{ \sum_{h\in[H,H]}f_{h}\lambda_{hg}\mid f_{h}\in F_{1}\text{ for all }h\in[H,H]\right\} 
\]
is an $\varepsilon$-net for the image of $\mathcal{S}_{g}$ in $(A\rtimes_{\alpha,r}G)/\mathbb{C}1$
if $g\notin[H,H]$. This finishes the proof. \end{proof}

\vspace{3mm}


\subsection{The construction of odd spectral triples\label{ConverseCase}}

In \cite[Subsection 2.4]{HSWZ13}, it was noted that analogous to the
construction of even spectral triples on crossed product C$^{\ast}$-algebras
coming from odd spectral triples, a similar procedure can be used
to obtain odd spectral triples coming from even ones. As in Subsection
\ref{subsec:Crossed-product-C-algebras}, let $\alpha:G\rightarrow\text{Aut}(A)$
be an action of a discrete group $G$ on a unital separable C$^{\ast}$-algebra
$A$ and let $\ell:G\rightarrow\mathbb{R}_{+}$ be a proper length
function on $G$. Assume that 
\[
\left(\mathcal{A},\mathcal{H}_{A,1}\oplus\mathcal{H}_{A,2},\left(\begin{array}{cc}
0 & D_{A}\\
D_{A}^{\ast} & 0
\end{array}\right)\right)
\]
is a spectral triple on $A$ with $\mathbb{Z}_{2}$-grading $\mathcal{H}_{A}:=\mathcal{H}_{A,1}\oplus\mathcal{H}_{A,2}$
and corresponding faithful representation $\pi:=\pi_{1}\oplus\pi_{2}$.
As before, consider the canonical odd spectral triple $(\mathbb{C}[G],\ell^{2}(G),M_{\ell})$
on $C_{r}^{\ast}(G)$, where $M_{\ell}$ denotes the multiplication
operator $\delta_{g}\mapsto\ell(g)\delta_{g}$ for $g\in G$. The
reduced crossed product C$^{\ast}$-algebra $A\rtimes_{\alpha,r}G$
can be (faithfully) represented on $\mathcal{H}:=(\mathcal{H}_{A,1}\otimes\ell^{2}(G))\oplus(\mathcal{H}_{A,2}\otimes\ell^{2}(G))$
in a natural way. By assuming \emph{metric equicontinuity} in the
sense that $\alpha_{g}(\mathcal{A})\subseteq\mathcal{A}$ and $\sup_{g\in G}\Vert\pi_{1}(g.a)D_{A}-D_{A}\pi_{2}(g.a)\Vert<\infty$
for all $a\in\mathcal{A}$, $g\in G$, one can define an odd spectral
triple $(C_{c}(G,\mathcal{A}),\mathcal{H},D)$ on $A\rtimes_{\alpha,r}G$,
where
\begin{equation}
D:=\left(\begin{array}{cc}
1\otimes M_{\ell} & D_{A}\otimes1\\
D_{A}^{\ast}\otimes1 & -1\otimes M_{\ell}
\end{array}\right).\label{eq:DiracOperator}
\end{equation}
This triple is non-degenerate if the one on $A$ is. \\

We claim that an analog to Theorem \ref{MainTheorem} holds in this setting as well. This
follows from a variation of the characterization in Proposition \ref{QuantumMetricSpaceCharacterization}.

\begin{proposition} \label{QuantumMetricSpaceCharacterization2}
The even spectral triple $(C_{c}(G,\mathcal{A}),\mathcal{H},D)$ defined
above is a spectral metric space if and only if the set of
all elements $x=\sum_{g\in G}a_{g}\lambda_{g}\in C_{c}(G,\mathcal{A})\subseteq\mathcal{B}(\mathcal{H})$
with $(a_{g})_{g\in G}\subseteq\mathcal{A}$ for which the operator norms of the commutators
\[
\left[x,\left(\begin{array}{cc}
1\otimes M_{\ell} & 0\\
0 & -1\otimes M_{\ell}
\end{array}\right)\right],\left[x,\left(\begin{array}{cc}
0 & D_{A}\otimes1\\
D_{A}^{\ast}\otimes1 & 0
\end{array}\right)\right]\in\mathcal{B}(\mathcal{H})
\]
are bounded by $1$, has totally bounded image in $(A\rtimes_{\alpha,r}G)/\mathbb{C}1$.
\end{proposition}

\begin{theorem} \label{MainTheorem2} Let
\[
\left(\mathcal{A},\mathcal{H}_{A,1}\oplus\mathcal{H}_{A,2},\left(\begin{array}{cc}
0 & D_{A}\\
D_{A}^{\ast} & 0
\end{array}\right)\right)
\]
be a non-degenerate even spectral triple on a separable unital C$^{\ast}$-algebra
$A$ with $\mathbb{Z}_{2}$-grading $\mathcal{H}_{A}:=\mathcal{H}_{A,1}\oplus\mathcal{H}_{A,2}$
and corresponding representation $\pi:=\pi_{1}\oplus\pi_{2}$, and
assume that the induced Lipschitz semi-norm $L_{D_{A}}(a):=\left\Vert [D_{A},a]\right\Vert ,a\in\mathcal{A}$
defines a compact quantum metric space $(A,L_{D_{A}})$. Let further
$\alpha:G\rightarrow\text{Aut}(A)$ be a metrically equicontinuous
action of a finitely generated discrete group $G$ equipped with a
proper length function $\ell:G\rightarrow\mathbb{R}_{+}$ and assume
that there exists a finite index subgroup $H$ of $G$ that is separated
with respect to the restricted length function $\ell|_{H}$. As before,
define $\mathcal{H}:=(\mathcal{H}_{A,1}\otimes\ell^{2}(G))\oplus(\mathcal{H}_{A,2}\otimes\ell^{2}(G))$
and $D$ as in \eqref{eq:DiracOperator}. Then the odd spectral triple
$(C_{c}(G,\mathcal{A}),\mathcal{H},D)$ is a spectral metric space
if and only if for every $g\in G$, the set of all elements $x=\sum_{h\in[H,H]}a_{h}\lambda_{hg}\in C_{c}(G,\mathcal{A})$
with $(a_{h})_{h\in[H,H]}\subseteq\mathcal{A}$ for which the operator norms of the commutators
\[
\left[x,\left(\begin{array}{cc}
1\otimes M_{\ell} & 0\\
0 & -1\otimes M_{\ell}
\end{array}\right)\right],\left[x,\left(\begin{array}{cc}
0 & D_{A}\otimes1\\
D_{A}^{\ast}\otimes1 & 0
\end{array}\right)\right]
\]
are bounded by $1$, has totally bounded image in $(A\rtimes_{\alpha,r}G)/\mathbb{C}1$.

In particular, if $[H,H]$ is finite, then $(C_{c}(G,\mathcal{A}),\mathcal{H},D)$
is a spectral metric space. \end{theorem}

Despite being lengthy, the arguments for proving Theorem \ref{MainTheorem2}
are essentially variations of those in Subsection \ref{subsec:Crossed-product-C-algebras}
and Subsection \ref{MainPart}. We therefore omit the details here. \\

\begin{remark} By applying \cite[Proposition 2.8]{HSWZ13} and its counterpart for
even spectral triples (see the discussion in \cite[Subsection 2.4]{HSWZ13}),
an iteration of Theorem \ref{MainTheorem} and Theorem \ref{MainTheorem2} 
allows to construct spectral triples on suitable crossed products
of the form $A\rtimes_{\alpha,r}G^{m}\cong(...((A\rtimes_{\alpha_{1},r}G)\rtimes_{\alpha_{2},r}G)...)\rtimes_{\alpha_{m},r}G$,
$m\in\mathbb{N}$ that give rise to quantum metric spaces; compare
with \cite[Theorem 2.14]{HSWZ13}. Here, as before, $G$ is a finitely
generated discrete group equipped with a proper length function $\ell:G\rightarrow\mathbb{R}_{+}$
that admits a finite index subgroup $H$ that is separated with respect
to $\ell|_{H}$ and whose commutator $[H,H]$ is finite. The $\alpha_{i}$,
$1 \leq i \leq m $ denote the (metrically equicontinuous) coordinate $G$-actions
of $\alpha$. \end{remark}

\vspace{3mm}


\section{Groups separated with respect to length functions\label{SeparatedGroups}}

Recall that we call a finitely generated discrete group $G$ separated
with respect to a pseudo-length function $\ell$ if $\text{Hom}(\text{im}(p_{G}),\mathbb{R})=\text{Span}\left\{ \widehat{\mu}_{\ell}\mid\mu\text{ invariant mean}\right\} $,
where $p_{G}$ is the projection onto the torsion-free component of
the Abelianization of $G$ and $\widehat{\mu}_{\ell}$ is given by
$(\widehat{\mu}_{\ell}\circ p_{G})(g):=\mu(\varphi_{g}^{\ell})$ for
$g\in G$. In the present section, we study the notion's link to the
asymptotic semi-norm construction in the Abelian setting and provide
groups and length functions that satisfy (a variation of) this quality. We further
give a counterexample that demonstrates that even the integers equipped
with a very natural length function are not separated.

\vspace{3mm}


\subsection{Integer lattices\label{subsec:IntegerLattices}}

Recall that a \emph{quasi-isometric embedding} between metric
spaces $(X,d_{X})$ and $(Y,d_{Y})$ is a map $f:X\rightarrow Y$
for which there exists $C\geq1$ and $r>0$ with
\[
C^{-1}d_{X}(x,y)-r\leq d_{Y}(f(x),f(y))\leq Cd(x,y)+r
\]
for all $x,y\in X$. It is well-known that, given finite generating
sets $S=S^{-1}$ and $S^{\prime}=(S^\prime)^{-1}$ of a group $G$, the identity map on $G$
equipped with the respective induced word metrics defines a quasi-isometric
embedding. Similarly, if $G$ is a finitely generated group and $H\leq G$
is a finite index subgroup, then $H$ is also finitely generated, and
the embedding of $H$ into $G$ is quasi-isometric with respect to
the induced word metrics; this follows, for instance, from the Milnor-Svarc
Lemma. We call two length functions $\ell,\ell^{\prime}:G\rightarrow\mathbb{R}_{+}$
on a group $G$ \emph{bi-Lipschitz equivalent} if the identity on
$G$ equipped with the metrics $d_{\ell}$ and $d_{\ell^{\prime}}$
is a quasi-isometric embedding; that is if there exist $C\geq1$ and
$r \geq 0$ with $C^{-1}\ell(g)-r\leq\ell^{\prime}(g)\leq C\ell(g)+r$
for all $g\in G$.

From now on, we restrict to the case of integer lattices $G=\mathbb{Z}^{m}$,
$m\in\mathbb{N}$ equipped with length functions $\ell:G\rightarrow\mathbb{R}_{+}$.
By applying Fekete's Subadditivity Lemma, one obtains that for every
$g\in G$, the limit $\lim_{i\rightarrow\infty}i^{-1}\ell(ig)$ exists
and that it coincides with $\inf_{i\in\mathbb{N}}i^{-1}\ell(ig)$;
so in particular, $\lim_{i\rightarrow\infty}i^{-1}\ell(ig)\leq\ell(g)$.
The function $g\mapsto\lim_{i\rightarrow\infty}i^{-1}\ell(ig)$ uniquely
extends to a semi-norm $\left\Vert \cdot\right\Vert _{\ell}$ on $\mathbb{R}^{m}$,
which is called the \emph{asymptotic semi-norm} (or \emph{stable semi-norm}) associated with $\ell$; see, for example, \cite[Proposition 8.5.3]{BBI01}. In many interesting cases,
the asymptotic semi-norm is positive definite (i.e., a genuine norm). This
is, for instance, the case if $\ell$ is bi-Lipschitz equivalent to
a word length function (e.g., if $\mathbb{Z}^{m}$ embeds as a finite
index subgroup into a larger group and $\ell$ is a restricted word
length function). Indeed, in that case, there exists a constant $C\geq1$
such that 
\[
C^{-1}\left\Vert x\right\Vert _{1}=C^{-1}\left\Vert x\right\Vert _{\ell_{1}}\leq\left\Vert x\right\Vert _{\ell}\leq C\left\Vert x\right\Vert _{\ell_{1}}=C\left\Vert x\right\Vert _{1}
\]
for every $x\in\mathbb{R}^{m}$. Here, $\left\Vert \cdot\right\Vert _{1}$
denotes the 1-norm on $\mathbb{R}^{m}$, and $\ell_{1}$ is the word length function associated with the canonical generating set of $\mathbb{Z}^m$.

The restriction $\ell^{\text{as}}:\mathbb{Z}^{m}\rightarrow\mathbb{R}_{+}$
of the asymptotic semi-norm to $\mathbb{Z}^{m}$ is a homogeneous
pseudo-length function. The proof of the following lemma is an easy
exercise.

\begin{lemma} \label{StrongConvergence} For every $g\in G$, the
sequence $(i^{-1}(\varphi_{ig}^{\ell}-\varphi_{ig}^{\ell^{\text{as}}}))_{i\in\mathbb{N}}\subseteq\mathcal{B}(\ell^{2}(G))$
strongly converges to $0$. \end{lemma}

In general, there is no reason to expect that the sequence in Lemma
\ref{StrongConvergence} converges with respect to the operator norm (i.e., uniformly). Still,
for many natural examples, that is the case.

As it turns out, $\ell^{\text{as}}$
very naturally occurs in the context of the question for separateness
of the pair $(G,\ell)$.

\begin{proposition} \label{OperatorConvergence} Let $\ell:G\rightarrow\mathbb{R}_{+}$
be a length function on $G=\mathbb{Z}^{m}$, $m\in\mathbb{N}$. Assume
that $i^{-1}(\varphi_{ig}^{\ell}-\varphi_{ig}^{\ell^{\text{as}}})\rightarrow0$
uniformly. Then $(G,\ell)$ is separated if and only if
$(G,\ell^{\text{as}})$ is separated. \end{proposition}

\begin{proof} Let $\mu:\ell^{\infty}(G)\rightarrow\mathbb{C}$ be
an invariant mean. Then,
\[
\widehat{\mu}_{\ell}\circ p_{G}(g)-\widehat{\mu}_{\ell^{\text{as}}}\circ p_{G}(g)=\mu(i^{-1}(\varphi_{ig}^{\ell}-\varphi_{ig}^{\ell^{\text{as}}}))\rightarrow0
\]
for every $g\in G$, and hence, $\widehat{\mu}_{\ell}=\widehat{\mu}_{\ell^{\text{as}}}$.
This implies the claim. \end{proof}

By adding the assumption that the asymptotic semi-norm is positive
definite, we obtain the following much stronger result. Before giving a proof, we pick up some of its implications.

\begin{theorem} \label{StrongerResult} Let $\ell:G\rightarrow\mathbb{R}_{+}$
be a length function on $G=\mathbb{Z}^{m}$, $m\in\mathbb{N}$. Assume
that $i^{-1}(\varphi_{ig}^{\ell}-\varphi_{ig}^{\ell^{\text{as}}})\rightarrow0$
uniformly and that the asymptotic semi-norm associated with
$\ell$ is positive definite. Then $G$ is separated with respect
to $\ell$. \end{theorem}

Theorem \ref{StrongerResult} applies to many natural situations.
If, for instance, $\ell$ is a word length function, it is easy to show
that the map $G\rightarrow\mathbb{R}_{+}$, $g\mapsto\left|\ell(g)-\left\Vert g\right\Vert _{\ell}\right|$
is bounded (see, for example, \cite[Lemma 3.5]{DLM12}), and hence, $i^{-1}(\varphi_{ig}^{\ell}-\varphi_{ig}^{\ell^{\text{as}}})\rightarrow0$
uniformly. The proof of the following more general statement
relies on the results in \cite{LOZ21}, which are again in the spirit
of Burago's approach in \cite{Burago}. Recall that an action of a
group $G$ on a metric space $(X,d)$ is called \emph{cocompact}
if the quotient space $X/G$ is compact. It is called \emph{properly
discontinuous} if each point admits a neighborhood satisfying the property
that all non-trivial elements of $G$ move the neighborhood outside
itself. The metric space $(X,d)$ is \emph{geodesic} if, given two
points, there exists a path between them whose length equals the distance
between the points. Here, the length of a path $c:[0,1]\rightarrow X$
is defined as the infimum over all sums $\sum_{i=1}^{k}d(c(t_{i-1}),c(t_{i}))$,
where $0\leq t_{0}\leq t_{1}...\le t_{k}\leq1$.

\begin{corollary} \label{OrbitMetric} Let $(X,d)$ be a proper, geodesic
metric space on which $\mathbb{Z}^{m}$, $m \in \mathbb{N}$ acts freely, cocompactly, and
properly discontinuously by isometries. Assume that there exists a
continuous map $F:X\rightarrow\mathbb{R}^{m}$ that is equivariant
with respect to the canonical shift action $\mathbb{Z}^{m}\curvearrowright\mathbb{R}^{m}$
and let $x_{0}\in X$. Define $\ell:\mathbb{Z}^{m}\rightarrow\mathbb{R}_{+}$
by $\ell(g):=d(x_{0},g.x_{0})$ for $g\in\mathbb{Z}^{m}$, where $x_{0}\in X$.
Then $\ell$ is a proper length function, and $G$ is separated with
respect to $\ell$.

In particular, if $G$ is a discrete group finitely generated by a
set $S$ with $S=S^{-1}$ that contains $\mathbb{Z}^{m}$, $m \in \mathbb{N}$ as a finite
index normal subgroup, then $\mathbb{Z}^m$ is separated with respect to the
restricted word length function $\ell_{S}|_{\mathbb{Z}^{m}}$. \end{corollary}

\begin{proof} For the first statement, recall that Fekete's Subadditivity
Lemma implies $\ell^{\text{as}}(h)\leq\ell(h)$ for all $h\in\mathbb{Z}^{m}$.
By \cite[Lemma 20]{LOZ21}, there further exists a constant $C\geq0$
such that $2\ell(h)\leq\ell(2h)+C$ for every $h\in\mathbb{Z}^{m}$.
Inductively, we obtain that
\begin{eqnarray}
\nonumber
\ell(h)\leq\frac{\ell(2h)}{2}+C\leq\frac{\ell(4h)}{4}+\frac{3C}{2}\leq...\leq\frac{\ell(2^{i}h)}{2^{i}}+\left(2-\frac{1}{2^{i-1}}\right)C
\end{eqnarray}
for all $h\in G$, $i\in\mathbb{N}$, and therefore, $\ell(h)\leq\ell^{\text{as}}(h)+2C$.
But then
\begin{eqnarray}
\nonumber
\left\Vert \frac{\varphi_{ig}^{\ell}-\varphi_{ig}^{\ell^{\text{as}}}}{i}\right\Vert \leq \sup_{h\in\mathbb{Z}^{m}}\left\{ \left|\frac{\ell(h)-\ell^{\text{as}}(h)}{i}\right|+\left|\frac{\ell(h-ig)-\ell^{\text{as}}(h-ig)}{i}\right|\right\} \leq\frac{4C}{i}\rightarrow0
\end{eqnarray}
for $g\in G$.

The asymptotic semi-norm associated with $\ell$ is further
positive definite. Indeed, by \cite[Theorem 8.3.19]{BBI01}, the metric
$(g,h)\mapsto d(g.x_{0},h.x_{0})$ on $G$ is bi-Lipschitz equivalent
to a word metric, and thus, $\ell$ is bi-Lipschitz equivalent to a
word length function. Therefore, by our discussion above, $\left\Vert \cdot\right\Vert _{\ell}$
is positive definite and $\ell$ is proper. We deduce the statement
of the first part of the corollary by invoking Theorem \ref{StrongerResult}. \\

For the second statement, we argue as in \cite[Corollary 23]{LOZ21}.
Consider $G$ equipped with the word length metric $d_{\ell_S}$. Then
$(G,d_{\ell_S})$ is a proper geodesic metric space, and the action of $\mathbb{Z}^{m}$
on $G$ via left translation is free, cocompact, and properly discontinuous.
Choose elements $g_{1},...,g_{k}\in G$ with $G=\bigcup_{i=1}^{k}\mathbb{Z}^{m}g_{i}$
and $\mathbb{Z}^{m}g_{i}\neq\mathbb{Z}^{m}g_{j}$ for $i\neq j$ and
define $F:G\rightarrow\mathbb{R}^{m}$ via $F(hg_{i}):=F(g_{i})+h$
for $h\in\mathbb{Z}^{m}$, $1\leq i\leq k$, where $F(g_{1}),...,F(g_{k})\in\mathbb{R}^{m}$
are chosen arbitrarily. Then $F$ satisfies the conditions of the
first part of the corollary, and hence, $G$ is separated with respect
to the restricted word length function $\ell_{S}|_{\mathbb{Z}^{m}}$.
\end{proof}

With Theorem \ref{MainTheorem} and Theorem \ref{MainTheorem2} at hand, the Corollary \ref{OrbitMetric} implies
the following important fact. Of course, the statement holds for all
orbit metrics as in Corollary \ref{OrbitMetric}.

\begin{corollary} \label{VirtuallyAbelian} The following two statements hold:
\begin{enumerate}
\item Let $(\mathcal{A},\mathcal{H}_{A},D_{A})$
be a non-degenerate odd spectral triple on a separable unital C$^{\ast}$-algebra
$A$ and assume that the induced Lipschitz semi-norm $L_{D_{A}}(a):=\left\Vert [D_{A},a]\right\Vert ,a\in\mathcal{A}$
defines a compact quantum metric space $(A,L_{D_{A}})$. Let further
$\alpha:G\rightarrow\text{Aut}(A)$ be a metrically equicontinuous
action of a virtually Abelian discrete group $G$ that is finitely
generated by a set $S$ with $S=S^{-1}$ and let $\ell:G\rightarrow\mathbb{R}_{+}$
be the corresponding word length function. Define $\mathcal{H}:=\mathcal{H}_{A}\otimes\ell^{2}(G)$
and $D$ as in Subsection \ref{subsec:Crossed-product-C-algebras}.
Then the even spectral triple $(C_{c}(G,\mathcal{A}),\mathcal{H}\oplus\mathcal{H},D)$
is a spectral metric space.
\item Let
\[
\left(\mathcal{A},\mathcal{H}_{A,1}\oplus\mathcal{H}_{A,2},\left(\begin{array}{cc}
0 & D_{A}\\
D_{A}^{\ast} & 0
\end{array}\right)\right)
\]
be a non-degenerate even spectral triple on a separable unital C$^{\ast}$-algebra
$A$ with $\mathbb{Z}_{2}$-grading $\mathcal{H}_{A}:=\mathcal{H}_{A,1}\oplus\mathcal{H}_{A,2}$
and corresponding representation $\pi:=\pi_{1}\oplus\pi_{2}$, and assume that the induced Lipschitz semi-norm $L_{D_{A}}(a):=\left\Vert [D_{A},a]\right\Vert ,a\in\mathcal{A}$
defines a compact quantum metric space $(A,L_{D_{A}})$. Let further
$\alpha:G\rightarrow\text{Aut}(A)$ be a metrically equicontinuous
action of a virtually Abelian discrete group $G$ that is finitely
generated by a set $S$ with $S=S^{-1}$ and let $\ell:G\rightarrow\mathbb{R}_{+}$
be the corresponding word length function. Define $\mathcal{H}:=(\mathcal{H}_{A,1}\otimes\ell^{2}(G))\oplus(\mathcal{H}_{A,2}\otimes\ell^{2}(G))$
and $D$ as in Subsection \ref{ConverseCase}.
Then the odd spectral triple $(C_{c}(G,\mathcal{A}),\mathcal{H},D)$
is a spectral metric space.
\end{enumerate}
 \end{corollary}

Let us now turn to the proof of Theorem \ref{StrongerResult}. Our
argument requires Rieffel's construction in \cite[Section 7]{Rieffel02}.
For a given norm $\left\Vert \cdot\right\Vert $ on $\mathbb{R}^{m}$,
write $\ell_{\Vert\cdot\Vert}$ for its restriction to $\mathbb{Z}^{m}$;
this defines a length function that is bi-Lipschitz equivalent to
the word length function $\ell_1$ from before. Let $v\in\mathbb{R}^{m}$ with $\left\Vert v\right\Vert =1$
be a \emph{smooth point} in the sense that there exists exactly one
functional $\sigma_{v}$ on $\mathbb{R}^{m}$ with $ \Vert \sigma_{v} \Vert =1=\sigma_{v}(v)$;
for the background on tangent functionals, see \cite[Section V.9]{DunfordSchwartz}.
Then the geodesic ray $\mathbb{R}_{+}\ni t\mapsto tv$ determines
a Busemann point $\mathfrak{b}_{v}$ in the horofunction boundary
$\partial_{\Vert\cdot\Vert}\mathbb{R}^{m}$, and by \cite[Proposition 6.2]{Rieffel02}
and \cite[Proposition 6.3]{Rieffel02}, this point is fixed under the
action of $\mathbb{R}^{m}$ with $\Vert tv \Vert - \Vert tv-x \Vert \rightarrow \sigma_v(x)$
for every $x\in\mathbb{R}^{m}$. By invoking a variation of Kronecker's
theorem (see \cite[Lemma 7.2]{Rieffel02}), one finds an unbounded
strictly increasing sequence $(t_{i})_{i\in\mathbb{N}_{\geq1}}\subseteq\mathbb{R}_{+}$
such that for every $i\in\mathbb{N}_{\geq1}$, there exists $x_{i}\in\mathbb{Z}^{m}$
with $\left\Vert x_{i}-t_{i}v\right\Vert <i^{-1}$. Set $t_{0}:=0$,
$x_{0}:=0$, and define $\gamma:\{t_{i}\mid i\in\mathbb{N}\}\rightarrow\mathbb{Z}^{m}$
by $\gamma(t_{i}):=x_{i}$. Then $\gamma$ is an almost geodesic ray
that determines a Busemann point $\mathfrak{b}_{v}^{\prime}\in\partial_{\ell_{\Vert\cdot\Vert}}\mathbb{Z}^{m}$.
By \cite[Proposition 7.4]{Rieffel02}, this point is fixed under the
action of $\mathbb{Z}^{m}$ and satisfies $\varphi_{g}^{\ell_{\Vert\cdot\Vert}}(\mathfrak{b}_{v}^{\prime})=\sigma_{v}(g)$
for every $g\in\mathbb{Z}^{m}$.

\begin{proof}[{Proof of Theorem \ref{StrongerResult}}] Following
Proposition \ref{OperatorConvergence}, it suffices to show that $G$
is separated with respect to the length function $\ell^{\text{as}}$. As $\ell^{\text{as}}$ is the
restriction of the asymptotic (semi-)norm $\left\Vert \cdot\right\Vert _{\ell}$,
we may apply \cite[Proposition 7.4]{Rieffel02} to find for every
smooth point $v\in\mathbb{R}^{m}$ (with respect to the asymptotic semi-norm) with $\left\Vert v\right\Vert _{\ell}=1$
a point $\mathfrak{b}_{v}^{\prime}\in\partial_{\ell^{\text{as}}}G$ that
is fixed under the action of $G$ with $\varphi_{g}^{\ell^{\text{as}}}(\mathfrak{b}_{v}^{\prime})=\sigma_{v}(g)$
for every $g\in G$. Evaluation in $\mathfrak{b}_{v}^{\prime}$ leads
to a (multiplicative) $G$-invariant state $\nu_{v}$ on $C(\overline{G}^{\ell^{\text{as}}})$.
Further, by the amenability of $G$, the linear map $\chi:C_{r}^{\ast}(G)\rightarrow\mathbb{C}$
defined by $\chi(\lambda_{g}):=1$ for $g\in G$ is bounded and multiplicative
(see \cite[Theorem 2.6.8]{BrownOzawa}). Recall that Proposition \ref{Isomorphism-1} (in combination with Fell's absorption principle; see \cite[Proposition 4.1.7]{BrownOzawa})
provides a canonical identification of $C(\overline{G}^{\ell^{\text{as}}})\rtimes_{\beta,r}G$
with the C$^{\ast}$-subalgebra of $\mathcal{B}(\ell^{2}(G))$ generated
by $C_{r}^{\ast}(G)$, $C_{0}(G)$ and the multiplication operators
$\{\varphi_{g}^{\ell^{\text{as}}}\mid g\in G\}$. Via composing $\chi$ with
the conditional expectation $C(\overline{G}^{\ell^{\text{as}}})\rtimes_{\beta,r}G\rightarrow C_{r}^{\ast}(G)$,
$f\lambda_{g}\mapsto\nu_{v}(f)\lambda_{g}$ for $f\in C(\overline{G}^{\ell^{\text{as}}})$,
$g\in G$ and extending to $\mathcal{B}(\ell^{2}(G))$, we hence obtain
a state that contains $C_{r}^{\ast}(G)$ in its multiplicative domain
(see \cite[Proposition 1.5.7]{BrownOzawa}). It is thus invariant
under the canonical action of $G$ and restricts to an invariant mean
$\mu_{v}:\ell^{\infty}(G)\rightarrow\mathbb{C}$ with $\widehat{(\mu_{v})}_{\ell^{\text{as}}}=\sigma_{v}|_{G}$.
To conclude the statement from the theorem it hence suffices to prove
that the span of all $\sigma_{v}|_G$ where $v\in\mathbb{R}^{m}$ is
a smooth point of the unit sphere (with respect to $\Vert \cdot \Vert_\ell$), coincides with $\text{Hom}(G,\mathbb{R})$.
For this purpose assume that the complement of the span is non-empty
and denote the canonical orthonormal basis of $\mathbb{R}^{m}$ by
$(e_{i})_{i=1,...,m}$. Then there exists a non-trivial vector $\xi$
in the orthogonal complement (with respect to the canonical inner
product on $\mathbb{R}^{m}$) of 
\[
\text{Span}\{(\sigma_{v}(e_{i}))_{i=1,...,m}\in\mathbb{R}^{m}\mid v\in\mathbb{R}^{m}\text{ with }\left\Vert v\right\Vert _{\ell}=1\text{ smooth point}\}\subseteq\mathbb{R}^{m}.
\]
But this means that $\sigma_{v}(\xi)=0$ for all smooth points $v\in\mathbb{R}^{m}$
of the unit sphere. With \cite[Proposition 6.7]{Rieffel02} we conclude
that $\xi=0$ in contradiction to our assumption that $\xi$ is non-trivial.
\end{proof}

\vspace{3mm}


\subsection{Nilpotent groups}

Besides constructing fixed points in the horofunction boundary
of $\mathbb{Z}^{m}$, $m\in\mathbb{N}$ associated with length functions
that are restrictions of norms on $\mathbb{R}^{m}$, in \cite[Section 8]{Rieffel02},
Rieffel also constructed finite orbits of horofunction boundaries
of $\mathbb{Z}^{m}$ associated with word length functions. The investigation
of such points was later extended by Walsh in \cite{Walsh07} to nilpotent
groups. He proved that for a given nilpotent group $G$ finitely generated
by a set $S$ with $S=S^{-1}$, there is one finite orbit associated
with each facet of the polytope obtained by projecting $S$ onto the
torsion-free component of the Abelianization of $G$. The aim of this
subsection is to discuss the implications of Walsh's results in our context.

Let us review the construction in \cite{Walsh07} in more detail. The map $p_{G}$
from before gives a group homomorphism $G\rightarrow\mathbb{Z}^{m}$,
where $m$ is the rank of $G/[G,G]$. Again, view $\mathbb{Z}^{m}$
as embedded into $\mathbb{R}^{m}$ and consider the convex hull $K_{S}:=\text{conv}(p_{G}(S))$.
The set $K_{S}$ defines a polytope in $\mathbb{R}^{m}$. Its proper
faces of co-dimension $1$ are called \emph{facets}. For such a facet
$F$, consider the subset $V_{F}:=\{s\in S\mid p_{G}(s)\in F\}$ of
$S$ and write $\left\langle V_{F}\right\rangle $ for the (nilpotent)
subgroup of $G$ generated by $V_{F}$. This subgroup has finite index
in $G$. Further, by \cite[Section 4]{Walsh07}, one finds a word $w_{F}$
with letters in $V$ such that the infinite reduced word $w_{F}w_{F}...$
defines a geodesic path in the Cayley graph of $G$ with respect to
$S$ in the sense that each of the word's prefixes are geodesic with
respect to the word metric $d_{\ell_S}$. By Theorem \ref{Busemann}, this geodesic
path gives a Busemann point $\xi_{F}\in\partial_{\ell_{S}}G$ in the
horofunction boundary of $G$. The stabilizer of $\xi_{F}$ is given
by $\left\langle V_{F}\right\rangle $. But even more is true.

\begin{theorem}[{\cite[Theorem 1.1]{Walsh07}}] Let $G$ be a nilpotent
group with finite generating set $S=S^{-1}$ and consider the action
of $G$ on its horofunction boundary with respect to the corresponding
word length metric. Then there exists a natural one-to-one correspondence
between the finite orbits of Busemann points and the facets of $K_{S}$.
\end{theorem}

Let $\mathcal{F}$ be the (finite) set of facets of $K_{S}$. Similar
to \cite[Section 8]{Rieffel02} (and similar to Subsection \ref{subsec:IntegerLattices}),
every facet $F\in\mathcal{F}$ of $K_{S}$ is characterized by the
fact that there exists a (unique) linear functional $\sigma_{F}$
on $\mathbb{R}^{m}$ with $\sigma_{F}\circ p_G(s)\leq1$ for all $s\in S$
and $F=\text{conv}(\{p_{G}(s)\mid s\in S\text{ with }\sigma_{F}\circ p_{G}(s)=1\})$.
Rieffel calls this the \emph{support functional} of $F$.

\begin{lemma} \label{EveluationIdentity} For every $F\in\mathcal{F}$
and $h\in\left\langle V_{F}\right\rangle $, the equality $\varphi_{h}^{\ell_{S}}(\xi_{F})=\sigma_{F}\circ p_{G}(h)$
holds. \end{lemma}

\begin{proof} By \cite[Lemma 4.3]{Walsh07}, there exists $i_{0}\in\mathbb{N}$
such that for all $i\geq i_{0}$, the element $h^{-1}w_{F}^{i}$ can
be written as a product of elements of $V_{F}$. As in the proof of
\cite[Lemma 4.1]{Walsh07}, one deduced that $\left|h^{-1}w_{F}^{i}\right|=\sigma_{F}\circ p_{G}(h^{-1}w_{F}^{i})$
and $\left|w_{F}^{i}\right|=\sigma_{F}\circ p_{G}(w_{F}^{i})$ for
$i\geq i_{0}$ so that
\[
\varphi_{h}(\xi_{F})=\lim_{i\rightarrow\infty}(\sigma_{F}\circ p_{G}(w_{F}^{i})-\sigma_{F}\circ p_{G}(h^{-1}w_{F}^{i}))=\sigma_{F}\circ p_{G}(h).
\]
\end{proof}

Lemma \ref{EveluationIdentity} implies that nilpotent groups equipped
with word length functions contain finite index subgroups that satisfy a property that is close to being
separated.

\begin{proposition} \label{SeparatenessVariant}Let $G$ be a finitely
generated discrete nilpotent group $G$ with finite generating set
$S=S^{-1}$. Then there exists a finite index subgroup $H$ of $G$
such that every group homomorphism $H\rightarrow\mathbb{R}$ that
vanishes on $H\cap[G,G]$ can be written as a linear combination of
maps of the form $h\mapsto\mu(\varphi_{h}^{\ell_{S}})$, where $\mu:\ell^{\infty}(G)\rightarrow\mathbb{C}$
is an $H$-invariant state. \end{proposition}

\begin{proof} For every $F\in\mathcal{F}$, the group $\left\langle V_{F}\right\rangle $
has finite index in $G$, so $H:=\bigcap_{F\in\mathcal{F}}\left\langle V_{F}\right\rangle $
is a subgroup of finite index as well. The evaluation maps $\mathcal{G}(G,\ell_{S})\cong C(\overline{G}^{\ell_{S}})\rightarrow\mathbb{C}$,
$f\mapsto f(\xi_{F})$ with $F\in\mathcal{F}$ extend to $H$-invariant
states on $\ell^{\infty}(G)$ that we denote by $\mu_{F}$. As in
Lemma \ref{InducedHomomorphism}, one obtains that $(\widehat{\mu_{F}})_{H}:\text{im}(p_{H})\rightarrow\mathbb{R}$
given by $p_{H}(h)\mapsto\mu_{F}(\varphi_{h}^{\ell_{S}})$ is a well-defined
group homomorphism. By Lemma \ref{EveluationIdentity},
\[
\widehat{(\mu_{F})}_{H}\circ p_{H}(h)=\mu_{F}(\varphi_{h}^{\ell_{S}})=\sigma_{F}\circ p_{G}(h)
\]
for every $h\in H$, and therefore, $\widehat{(\mu_{F})}_{H}\circ p_{H}=(\sigma_{F}\circ p_{G})|_{H}$.
Now, every group homomorphism $\mathbb{Z}^{m}\rightarrow\mathbb{R}$
canonically extends to $\mathbb{R}^{m}$. Further, every group homomorphism
in $\text{Hom}(\mathbb{R}^{m},\mathbb{R})$ can be written as a linear
combination of support functionals $\sigma_{F}$, $F\in\mathcal{F}$.
Indeed, assume that $\text{Hom}(\mathbb{R}^{m},\mathbb{R})\setminus\text{Span}\{\sigma_{F}\mid F\in\mathcal{F}\}$
is non-empty and denote the canonical orthonormal basis of $\mathbb{R}^{m}$
by $(e_{i})_{i=1,...,m}$. Then there exists a non-trivial vector
$v$ in the orthogonal complement (with respect to the canonical inner product) of $\text{Span}\{(\sigma_{F}(e_{i}))_{i=1,...,m}\in\mathbb{R}^{m}\mid F\in\mathcal{F}\}\subseteq\mathbb{R}^{m}$.
Without loss of generality, we can assume that $v$ is contained in
some facet $F\in\mathcal{F}$. But then 
\[
1=\sigma_{F}(v)=\left\langle v,(\sigma_{F}(e_{i}))_{i=1,...,m}\right\rangle =0,
\]
which is a contradiction. Hence, $\text{Span}\{\sigma_{F}\mid F\in\mathcal{F}\}=\text{Hom}(\mathbb{R}^{m},\mathbb{R})$
as claimed. By using this, we obtain that every group homomorphism
$H\rightarrow\mathbb{R}$ that vanishes on $H\cap[G,G]$ can be written
as a linear combination of the maps $\widehat{(\mu_{F})}_{H}\circ p_{H}$,
$F\in\mathcal{F}$. \end{proof}

It can be checked that the property in Proposition \ref{SeparatenessVariant}
allows to prove a variant of Proposition \ref{DiagonalApproximation}
and Theorem \ref{MainTheorem} for nilpotent groups. Since the following theorem does not
lead to interesting new examples of quantum metric spaces and since
the proof is similar to the one in Subsection \ref{MainPart}, we
leave the details to the reader.

\begin{theorem} \label{NilpotentTheorem} Let $(\mathcal{A},\mathcal{H}_{A},D_{A})$
be a non-degenerate odd spectral triple on a separable unital C$^{\ast}$-algebra
$A$ and assume that the induced Lipschitz semi-norm $L_{D_{A}}(a):=\left\Vert [D_{A},a]\right\Vert ,a\in\mathcal{A}$
defines a compact quantum metric space $(A,L_{D_{A}})$. Let further
$\alpha:G\rightarrow\text{Aut}(A)$ be a metrically equicontinuous
action of a finitely generated discrete nilpotent group $G$ equipped
with a word length function. As before, define $\mathcal{H}:=\mathcal{H}_{A}\otimes\ell^{2}(G)$
and $D$ as in Subsection \ref{subsec:Crossed-product-C-algebras}.
Then the even spectral triple $(C_{c}(G,\mathcal{A}),\mathcal{H}\oplus\mathcal{H},D)$
is a spectral metric space if and only if for every $h\in G$,
the set of all elements $x=\sum_{g\in[G,G]}a_{g}\lambda_{gh}\in C_{c}(G,\mathcal{A})$
with $(a_{g})_{g\in[G,G]}\subseteq\mathcal{A}$ satisfying $\left\Vert [D_{A}\otimes1,x]\right\Vert \leq1$
and $\left\Vert [1\otimes M_{\ell},x]\right\Vert \leq1$ has totally
bounded image in $(A\rtimes_{\alpha,r}G)/\mathbb{C}1$. \end{theorem}

\begin{example} Consider the discrete Heisenberg group 
\[
H_{3}:=\left\{ \left(\begin{array}{ccc}
1 & x & z\\
0 & 1 & y\\
0 & 0 & 1
\end{array}\right)\mid x,y,z\in\mathbb{Z}\right\} 
\]
and define elements
\[
a:=\left(\begin{array}{ccc}
1 & 1 & 0\\
0 & 1 & 0\\
0 & 0 & 1
\end{array}\right),\hspace*{1em}b:=\left(\begin{array}{ccc}
1 & 0 & 0\\
0 & 1 & 1\\
0 & 0 & 1
\end{array}\right),\hspace*{1em}c:=\left(\begin{array}{ccc}
1 & 0 & 1\\
0 & 1 & 0\\
0 & 0 & 1
\end{array}\right).
\]
The set $S:=\{a,a^{-1},b,b^{-1}\}$ generates $H_{3}$. Let $\ell:=\ell_{S}$
be the corresponding word length function. The commutator subgroup
of $H_{3}$ coincides with its center which is given by the cyclic
group $\left\langle c\right\rangle \cong\mathbb{Z}$. Let $(\mathcal{A},\mathcal{H}_{A},D_{A})$,
$L_{D_{A}}$, $\alpha$, $\mathcal{H}$, and $D$ be as in Theorem \ref{NilpotentTheorem}.
Then the theorem implies that $(C_{c}(G,\mathcal{A}),\mathcal{H}\oplus\mathcal{H},D)$
is a spectral metric space if and only if for every $h\in H_{3}$,
the set of all elements $x=\sum_{i=0}^{\infty}a_{i}\lambda_{c^{i}h}\in C_{c}(H_{3},\mathcal{A})$
with $(a_{i})_{i\in\mathbb{N}}\subseteq\mathcal{A}$ satisfying $\left\Vert [D_{A}\otimes1,x]\right\Vert \leq1$
and $\left\Vert [1\otimes M_{\ell},x]\right\Vert \leq1$ has totally
bounded image in $(A\rtimes_{\alpha,r}G)/\mathbb{C}1$. It would be
interesting to see if our methods can be extended to give an analog
to Corollary \ref{VirtuallyAbelian} in this setting (or even for general
nilpotent groups). However, $\ell(c^{i})=2\lceil 2\sqrt{|i|}\rceil $
for every $i\in\mathbb{N}$ by \cite{Blachere}. From this, it can
be deduced that every invariant mean $\ell^{\infty}(\left\langle c\right\rangle )\rightarrow\mathbb{C}$
must vanish on the elements $\varphi_{g}^{\ell|_{\left\langle c\right\rangle }}$,
$g\in \left\langle c\right\rangle $. The restriction of $\ell$ to $\left\langle c\right\rangle \cong\mathbb{Z}$
hence provides a natural example of a group and a length function on it
that is not separated. \end{example}

\vspace{3mm}


\section{Examples}

\vspace{3mm}

In this section we discuss natural examples of crossed products, that are covered by the results
of the previous sections. Our selection extends the one in \cite{HSWZ13}. 

\vspace{3mm}


\subsection{Actions on AF-algebras}

We begin by reminding the reader of a general construction by Christensen
and Ivan \cite{CI06} that allows associating spectral metric spaces with AF-algebras. This construction
was also employed in \cite[Subsection 3.3]{HSWZ13}. \\

Recall that an \emph{AF-algebra} is an inductive limit of a sequence
of finite-dimensional C$^{\ast}$-algebras. Given a unital AF-algebra
$A$, let $(\mathcal{A}_{i})_{i\in\mathbb{N}}\subseteq A$ with $\mathcal{A}_{0}:=\mathbb{C}1$
and $A=\overline{\bigcup_{i\in\mathbb{N}}\mathcal{A}_{i}}$ be an
increasing sequence of finite-dimensional C$^{\ast}$-algebras and
let $\phi$ be a faithful state on $A$. We call $(\mathcal{A}_{i})_{i\in\mathbb{N}}$
an \emph{AF-filtration}. Write $\pi_{\phi}$ for the (faithful) GNS-representation
of $A$ associated with $\phi$ and denote the corresponding GNS-Hilbert
space by $L^{2}(A,\phi)$. We will further write $\Omega_{\phi}\in L^{2}(A,\phi)$
for the canonical cyclic vector. Using this data, one can define a
sequence $(H_{i})_{i\in\mathbb{N}}$ of pairwise orthogonal finite-dimensional
subspaces of $L^{2}(A,\phi)$ via $H_{0}:=\pi_{\phi}(A_{0})\Omega_{\phi}=\mathbb{C}\Omega_{\phi}$
and $H_{i}:=\pi_{\phi}(A_{i})\Omega_{\phi}\cap(\pi_{\phi}(A_{i-1})\Omega_{\phi})^{\perp}$
for $i\in\mathbb{N}_{\geq1}$. Write $Q_{i}$, $i\in\mathbb{N}$ for
the orthogonal projection onto $H_{i}$. As was argued in \cite[Theorem 2.1]{CI06},
there exists a sequence $(\alpha_{i})_{i\in\mathbb{N}}$ of real numbers
with $\alpha_{0}=0$ and $|\alpha_{i}|\rightarrow\infty$ such that
the odd spectral triple $(\mathcal{A},L^{2}(A,\phi),D)$ with $\mathcal{A}:=\bigcup_{i\in\mathbb{N}}A_{i}$
and $D:=\sum_{i\in\mathbb{N}}\lambda_{i}Q_{i}$ is a spectral metric space.

Now assume that $\alpha:G\rightarrow\text{Aut}(A)$ is an action of
a discrete group $G$ on $A$, satisfying $\alpha_{g}(A_{i})\subseteq A_{i}$
for every $g\in G$, $i\in\mathbb{N}$. Since the elements of $A_{i}$
commute with the projections $Q_{j}$ for $j>i$, we obtain that for
$x\in A_{i}$,
\[
\sup_{g\in G}\Vert[D,\alpha_{g}(x)]\Vert=\sup_{g\in G}\Vert\sum_{j\leq i}\lambda_{j}[Q_{j},\alpha_{g}(x)]\Vert\leq\sum_{j\leq i}2\left|\lambda_{j}\right|\Vert x\Vert < \infty,
\]
that is, the action $\alpha$ is metrically equicontinuous. In particular,
if $G$ is a virtually Abelian group that is finitely generated by
a set $S$ with $S=S^{-1}$ and if $\ell:G\rightarrow\mathbb{R}_{+}$
is the corresponding word length function (or more generally a length
function as in Corollary \ref{OrbitMetric}), we obtain from Corollary
\ref{VirtuallyAbelian} a spectral metric space on the
crossed product $A\rtimes_{r,\alpha}G$.

\begin{example} Let $G$ be a countable residually finite discrete
group and let $(G_{i})_{i\in\mathbb{N}}$ be a strictly decreasing
sequence of finite index subgroups of $G$ with $\bigcap_{i\in\mathbb{N}}G_{i}=\{e\}$.
For every $i\in\mathbb{N}$ the group $G$ acts on $G/G_{i}$ via
left multiplication. Let $p_{i}:G/G_{i+1}\rightarrow G/G_{i}$ be
the (surjective and $G$-equivariant) map $gG_{i+1}\mapsto gG_{i}$
and consider the corresponding inverse limit $X$ given by 
\[
X:=\{(g_{i})_{i\in\mathbb{N}}\mid p_{i}(g_{i+1})=g_{i}\text{ for all }i\geq0\}\subseteq\prod_{i\in\mathbb{N}}G/G_{i}.
\]
We equip $X$ with the subspace topology of the product $\prod_{i\in\mathbb{N}}G/G_{i}$,
where each $G/G_{i}$, $i\in\mathbb{N}$ carries the discrete topology.
In this way, $X$ becomes a Cantor set, and the action of $G$ on its
left cosets extends to a continuous action on $X$ that (following
\cite[Definition 2]{CP08}; see also \cite{Krieger10}) we call
a \emph{$G$-subodometer action}. The commutative C$^{\ast}$-algebra
$C(X)$ identifies with the inductive limit $\lim_{\rightarrow}(C(G/G_{i}),\iota_{i})$,
where $\iota_{i}:C(G/G_{i})\rightarrow C(G/G_{i+1})$ is given by
$f\mapsto f\circ p$; so in particular, $C(X)$ is an AF-algebra. By
fixing a faithful state $\phi$ on $A:=C(X)$ and by setting $\mathcal{A}_{i}:=C(G/G_{i})$
for $i\in\mathbb{N}$, we can apply the construction from above to
obtain a spectral metric space $(\mathcal{A},L^{2}(C(X),\phi),D)$
on $C(X)$. In particular,
if $G$ is a virtually Abelian group that is finitely generated by
a set $S$ with $S=S^{-1}$ and if $\ell:G\rightarrow\mathbb{R}_{+}$
is the corresponding word length function (or more generally a length
function as in Corollary \ref{OrbitMetric}), we obtain a spectral metric space on $C(X)\rtimes_{r,\alpha}G$. The crossed product $C(X)\rtimes_{r,\alpha}G$
is called a \emph{generalized Bunce-Deddens algebra} (see \cite{Orfanos10}
and \cite{Carrion11}); note that for $G=\mathbb{Z}$ and $G_{i}:=(m_{1}...m_{i})\mathbb{Z}\subseteq\mathbb{Z}$
where $(m_{i})_{i\in\mathbb{N}}$ is a sequence of natural numbers
with $m_{i}\geq2$ for all $i\in\mathbb{N}$, we recover the classical
\emph{Bunce-Deddens algebras} (see \cite{BD75} and also \cite[Chapter V.3]{Davidson96}).
\end{example}

\vspace{3mm}


\subsection{Higher-dimensional non-commutative tori}

The \emph{rotation algebra} (or \emph{non-commutative 2-torus}) $\mathcal{A}_{\theta}$,
$\theta\in\mathbb{R}$, introduced in \cite{Rieffel81}, can be defined
as the universal C$^{\ast}$-algebra generated by two unitaries $u$
and $v$ subject to the relation $uv=e^{2\pi i\theta}vu$. In the
case where $\theta\in\mathbb{Z}$, $\mathcal{A}_{\theta}\cong C(\mathbb{T}^{2})$
and for irrational values of $\theta$, the C$^{\ast}$-algebra $\mathcal{A}_{\theta}$
is simple. The construction admits a natural generalization to higher
dimensions: let $\Theta:=(\theta_{i,j})_{i,j=1,...,m}$ be a skew
symmetric real $(m\times m)$-matrix (i.e., $\theta_{i,j}=-\theta_{j,i}$
for all $1\leq i,j\leq m$) and define $\mathcal{A}_{\Theta}$ to
be the universal C$^{\ast}$-algebra generated by unitaries $u_{1},...,u_{m}$
subject to relations $u_{i}u_{j}=e^{2\pi i\theta_{i,j}}u_{j}u_{i}$
for $1\leq i,j\leq m$. These C$^{\ast}$-algebras, which were defined
in \cite{Rieffel90}, are called \emph{non-commutative $m$-tori}.
Note that for $m=1$, the C$^{\ast}$-algebra $\mathcal{A}_{\Theta}$
is isomorphic to $C(\mathbb{T})$, and for $m=2$, we have $\mathcal{A}_{\Theta}=\mathcal{A}_{\theta_{1,2}}$.

Any non-commutative torus can be constructed as an iteration of crossed
products by actions of the integers $\mathbb{Z}$.
To make this precise, set $\Theta_{d}:=(\theta_{i,j})_{1\leq i,j\leq d}$ for $d=1,...,m$ and define an action 
\[
\alpha_{d}:\mathbb{Z}\curvearrowright\mathcal{A}_{\Theta_{d}}\text{ via }\alpha_{d}^n(u_{i}):=e^{-2\pi in\theta_{i,d+1}}u_{i},
\]
where the $u_{1},...,u_{d}$ are the standard generators of $\mathcal{A}_{\Theta_{d}}$.
Write $\widetilde{u}_{1},...,\widetilde{u}_{d-1}$ for the standard
generators of $\mathcal{A}_{\Theta_{d-1}}$. Then there exists an
isomorphism $\mathcal{A}_{\Theta_{d}}\cong\mathcal{A}_{\Theta_{d-1}}\rtimes_{\alpha_{d-1},r}\mathbb{Z}$
defined by $u_{i}\mapsto\widetilde{u}_{i}$ for $i=1,...,d-1$ and
$u_{d}\mapsto\lambda_{1}$. We obtain that
\[
\mathcal{A}_{\Theta}\cong(...((\mathcal{A}_{\Theta_{1}}\rtimes_{\alpha_{1}}\mathbb{Z})\rtimes_{\alpha_{2}}\mathbb{Z})...)\rtimes_{\alpha_{m-1}}\mathbb{Z}\cong(...((C(\mathbb{T})\rtimes\mathbb{Z})\rtimes\mathbb{Z})...)\rtimes\mathbb{Z},
\]
where the induced action of $\mathbb{Z}$ on $\mathbb{T}$ is given
by rotation by the angle $\theta_{1,1}$. Endow $C(\mathbb{T})$ with
the canonical non-degenerate odd spectral triple $(C^{\infty}(\mathbb{T}),L^{2}(\mathbb{T}),D)$,
where $D$ is the differentiation operator. We claim that, if we equip
the integers with word length functions (or more generally length
functions as in Corollary \ref{OrbitMetric}), a repeated application
of Corollary \ref{VirtuallyAbelian} leads to spectral metric spaces on the non-commutative $m$-tori. Since it is well-known that the spectral triple $(C^{\infty}(\mathbb{T}),L^{2}(\mathbb{T}),D)$
on $\mathcal{A}_{\Theta_{1}}\cong C(\mathbb{T})$ is metric, for this, it suffices to prove that for every $d=1,...,m-1$,
the action $\alpha_{d}:\mathbb{Z}\curvearrowright\mathcal{A}_{\Theta_{d}}$
is metrically equicontinuous. This can be proved via induction over
$d$: \\

For $d=1$, the action $\alpha_{d}:\mathbb{Z}\curvearrowright\mathcal{A}_{\Theta_{1}}\cong C(\mathbb{T})$
is obviously metrically equicontinuous. For the induction step, fix
$1\leq d\leq m-2$ and assume that the action $\alpha_{d}:\mathbb{Z}\curvearrowright\mathcal{A}_{\Theta_{d}}$
is metrically equicontinuous. We proceed by distinguishing two cases: \\

\begin{itemize}
\item \emph{Case }1: If $d$ is odd, the corresponding spectral triple on
$\mathcal{A}_{\Theta_{d}}$ is of the form $\left(\mathcal{A},\mathcal{H},D\right)$
with dense $\ast$-subalgebra $\mathcal{A}\subseteq\mathcal{A}_{\Theta_{d}}$
and corresponding faithful representation $\pi$. One easily checks
that $\alpha_{d+1}$ leaves both $\mathcal{A}$ and $C_{c}(\mathbb{Z},\mathcal{A})\subseteq\mathcal{A}_{\Theta_{d}}\rtimes_{\alpha_{d},r}\mathbb{Z}\cong\mathcal{A}_{\Theta_{d+1}}$
invariant. Further, if we denote the length function on $\mathbb{Z}$
by $\ell$,
\begin{eqnarray}
\nonumber
& & \sup_{n\in\mathbb{Z}}\left\Vert \left[\left(\begin{array}{cc}
0 & D\otimes1-i\otimes M_{\ell}\\
D\otimes1+i\otimes M_{\ell} & 0
\end{array}\right),\left(\begin{array}{cc}
\alpha_{d+1}^{n}(x) & 0\\
0 & \alpha_{d+1}^{n}(x)
\end{array}\right)\right]\right\Vert  \\
&\leq& \sup_{n\in\mathbb{N}}\left\{ 2\Vert[D\otimes1,\alpha_{d+1}^{n}(x)]\Vert+ 2\Vert[1\otimes M_{\ell},\alpha_{d+1}^{n}(x)]\Vert\right\} \label{eq:Supremum2}
\end{eqnarray}
 for every $x\in C_{c}(\mathbb{Z},\mathcal{A})$, $n\in\mathbb{Z}$.
For $x=\sum_{g\in\mathbb{Z}}a_{g}\lambda_{g}\in C_{c}(\mathbb{Z},\mathcal{A})$
with $(a_{g})_{g\in\mathbb{Z}}\subseteq\mathcal{A}$ we have $\alpha_{d+1}^{n}(x)=\sum_{g\in\mathbb{Z}}e^{-2\pi ing\theta_{d+1,d+2}}\alpha_{d+1}^{n}(a_{g})\lambda_{g}$
for all $n\in\mathbb{Z}$, and hence,
\begin{eqnarray}
\nonumber
\Vert[1\otimes M_{\ell},\alpha_{d+1}^{n}(x)]\Vert &=& \Vert\sum_{g\in\mathbb{Z}}e^{-2\pi ing \theta_{d+1,d+2}}\alpha_{d+1}^{n}(a_{g})[1\otimes M_{\ell},\lambda_{g}]\Vert \\
\nonumber
&\leq& \sum_{g\in\text{supp}(x)}\left\Vert a_{g}\right\Vert \left\Vert [1\otimes M_{\ell},\lambda_{g}]\right\Vert 
\end{eqnarray}
and 
\begin{eqnarray}
\nonumber
\Vert[D\otimes1,\alpha_{d+1}^{n}(x)]\Vert &=& \Vert\sum_{g\in\mathbb{Z}}e^{-2\pi ing \theta_{d+1,d+2}}[D\otimes1,\alpha_{d+1}^{n}(a_{g})]\lambda_{g}\Vert \\
\nonumber
&\leq& \sum_{g\in\text{supp}(x)}\Vert[D\otimes1,\alpha_{d+1}^{n}(a_{g})]\Vert.
\end{eqnarray}
Since $\Theta$ was arbitrary, it follows from the induction assumption
that the restriction of the action $\alpha_{d+1}$ to $\mathcal{A}_{\Theta_{d}}$
is metrically equicontinuous and hence that the supremum in \eqref{eq:Supremum2}
is finite. We deduce the metric equicontinuity of the action $\alpha_{d+1}:\mathbb{Z}\curvearrowright\mathcal{A}_{\Theta_{d+1}}$.
\item \emph{Case }2: If $d$ is even, the corresponding spectral triple
on $\mathcal{A}_{\Theta_{d}}$ is (since it is obtained by repeated
application of Corollary \ref{VirtuallyAbelian}) of the form 
\[
\left(\mathcal{A},\mathcal{H}\oplus\mathcal{H},\left(\begin{array}{cc}
0 & D\\
D^{\ast} & 0
\end{array}\right)\right)
\]
with dense $\ast$-subalgebra $\mathcal{A}\subseteq\mathcal{A}_{\Theta_{d}}$
and corresponding faithful representation $\pi\oplus\pi$. Again,
$\alpha_{d+1}$ leaves $\mathcal{A}$ and $C_{c}(\mathbb{Z},\mathcal{A})\subseteq\mathcal{A}_{\Theta_{d}}\rtimes_{\alpha_{d},r}\mathbb{Z}\cong\mathcal{A}_{\Theta_{d+1}}$
invariant. Further,
\begin{eqnarray}
\nonumber
& & \sup_{n\in\mathbb{Z}}\left\Vert \left[\left(\begin{array}{cc}
1\otimes M_{\ell} & D\otimes1\\
D^{\ast}\otimes1 & -1\otimes M_{\ell}
\end{array}\right),\left(\begin{array}{cc}
\alpha_{d+1}^{n}(x) & 0\\
0 & \alpha_{d+1}^{n}(x)
\end{array}\right)\right]\right\Vert \\
\nonumber
&\leq& \sup_{n\in\mathbb{Z}}\left\{ 2\Vert[1\otimes M_{\ell},\alpha_{d+1}^{n}(x)]\Vert+\Vert[D\otimes1,\alpha_{d+1}^{n}(x)]\Vert+\Vert[D^{\ast}\otimes1,\alpha_{d+1}^{n}(x)]\Vert\right\}
\end{eqnarray}
 for every $x\in C_{c}(\mathbb{Z},\mathcal{A})$, $n\in\mathbb{Z}$.
In the same way as before, one deduces that the action $\alpha_{d+1}:\mathbb{Z}\curvearrowright\mathcal{A}_{\Theta_{d+1}}$
is metrically equicontinuous.
\end{itemize}

\vspace{3mm}


\end{document}